\documentclass{amsart}

\usepackage{mathtools}
\usepackage{amsmath,amssymb,amsthm,amsfonts,mathrsfs}
\usepackage{graphpap,paralist,color}
\usepackage[mathscr]{eucal}
\usepackage{mathabx}
\usepackage[pdftex,colorlinks,citecolor=blue]{hyperref}
\usepackage{yfonts}
\usepackage{enumerate}
\usepackage[alphabetic]{amsrefs}
\usepackage{subfigure}

\usepackage[margin=1in]{geometry}

\numberwithin{equation}{section}

\theoremstyle{definition}
\newtheorem{theorem}{Theorem}[section]
\newtheorem{definition}[theorem]{Definition}

\newtheorem{lemma}[theorem]{Lemma}
\newtheorem{proposition}[theorem]{Proposition}
\newtheorem{corollary}[theorem]{Corollary}
\newtheorem*{theorem*}{Theorem}
\newtheorem*{TT1}{Theorem~\ref{thm:gmobiledetdone}}
\newtheorem*{TT2}{Theorem~\ref{thm:qmajgmobiledetdone}}
\theoremstyle{remark}
\newtheorem{remark}[theorem]{Remark}
\newtheorem{example}[theorem]{Example}

\newcommand\qbin[3]{\left[\begin{matrix} #1 \\ #2 \end{matrix} \right]_{#3}}
\newcommand\txtqbin[3]{\left[\begin{smallmatrix} #1 \\ #2 \end{smallmatrix} \right]_{#3}}

\newcommand{\stat}[0]{\text{stat}}
\newcommand{\maj}[0]{\text{maj}}
\newcommand{\inv}[0]{\text{inv}}
\newcommand{\Des}[0]{\text{Des}}

\DeclareMathOperator{\op}{op}

\newcommand{\cL}{\mathscr{L}}
\newcommand{\cP}{\mathcal{P}}
\newcommand{\cQ}{\mathcal{Q}}
\newcommand{\cT}{\mathcal{T}}
\newcommand{\cZ}{\mathcal{Z}}
\newcommand{\cR}{\mathcal{R}}

\newcommand{\co}{\colon}

\title{Counting linear extensions of posets with determinants of hook lengths}

\author{Alexander Garver}
\address{Department of Mathematics, University of Michigan, Ann Arbor, MI.}
\email{alexander.garver@gmail.com}
\author{Stefan Grosser}
\address{Department of Mathematics and Statistics, McGill University, Montreal, QC.}
\email{stefan.grosser@mail.mcgill.ca}
\author{Jacob P. Matherne}
\address{Department of Mathematics, University of Oregon, Eugene, OR, and Max-Planck-Institut f\"{u}r Mathematik, Bonn, Germany.}
\email{matherne@uoregon.edu}
\author{Alejandro H. Morales}
\address{Department of Mathematics and Statistics, University of Massachusetts, Amherst, MA.}
\email{ahmorales@math.umass.edu}

\thanks{Alexander Garver received support from NSERC grant RGPIN/05999-2014 and the Canada Research Chairs program. Jacob Matherne received support from NSF Grant DMS-1638352, the Association of Members of the Institute for Advanced Study, and the Max Planck Institute for Mathematics in Bonn. Alejandro Morales received support from NSF Grant DMS-1855536.}

\begin{document}

\begin{abstract}
We introduce a class of posets, which includes both ribbon posets (skew shapes) and $d$-complete posets, such that their number of linear extensions is given by a determinant of a matrix whose entries are products of hook lengths. We also give $q$-analogues of this determinantal formula in terms of the major index and inversion statistics.  As applications, we give families of tree posets whose numbers of linear extensions are given by generalizations of Euler numbers, we draw relations to Naruse--Okada's positive formulas for the number of linear extensions of skew $d$-complete posets, and we give polynomiality results analogous to those of descent polynomials by  Diaz-L\'opez, Harris, Insko, Omar, and Sagan.
\end{abstract}

\maketitle

\section{Introduction}

Linear extensions of posets are fundamental objects in combinatorics and computer science. The number of linear extensions  of a poset $\cP$, denoted by $e(\cP)$, is a measure of the complexity of the poset. However, computing $e(\cP)$ is a difficult problem---it is $\#P$-complete \cite{bw91}, even for posets with restricted height or dimension \cite{dp18}. Fortunately, for some posets that appear in algebraic and enumerative combinatorics, their number of linear extensions can be efficiently computed through product formulas (posets arising from Young diagrams \cite{frame1954hook}, rooted tree posets \cite{knuth_1998}, $d$-complete posets \cite{ProHLF}),
determinants (posets arising from skew Young diagrams \cite{aitken1943xxvi}), or recursive algorithms (series parallel posets \cite{mohring1989computationally}, tree posets \cite{atk}). 

The goal of this paper is to introduce and study a family of posets, called mobile posets, that is a common refinement of both ribbon posets (skew shapes) and $d$-complete posets.  We introduce a folding algorithm on such posets that allows us to compute their number of linear extensions via the determinant of a matrix whose entries are products of hook lengths (Theorem~\ref{thm:gmobiledetdone}).  Moreover, we obtain a $q$-analogue of this determinantal formula (Theorem~\ref{thm:qmajgmobiledetdone}).  These determinantal formulas simultaneously specialize to known formulas for both ribbon posets and $d$-complete posets in the literature.

\subsection{Hook-length formulas and determinantal formulas for \texorpdfstring{$e(\cP)$}{e(P)}}

The family of $d$-complete posets defined by Proctor \cites{ProClass,p99,ProHLF} includes rooted tree posets (posets whose Hasse diagram is a rooted tree) and posets from Young diagrams of (shifted) partitions. Proctor showed that these posets have a product formula involving {\em hook lengths} for their number of linear extensions:
\begin{equation}\label{eqn:introprod}
e(\cP) \,=\, \frac{n!}{\prod_{x \in \cP} h_{\cP}(x)},
\end{equation}
where $n$ is the number of elements of $\cP$, and $h_{\cP}(x)$ is the size of the {\em hook} of $x$, certain elements smaller than or equal to $x$ in $\cP$ (see Section~\ref{section:d-complete}). This formula generalizes the hook-length formula for $e(\cP)$ of a rooted tree due to Knuth \cite{knuth_1998}, as well as the hook-length formula for the number of standard tableaux of a (shifted) partition due to Frame--Robinson--Thrall \cites{Thrall, frame1954hook}. 

Tree posets are a natural generalization of rooted tree posets. The number of linear extensions of certain tree posets are of interest in enumerative and algebraic combinatorics. For example, given a set $S = \{s_1,\ldots,s_k\} \subset \{1,\ldots,n-1\}$ with $s_1<\cdots < s_k$, the number of permutations of size $n$ with descent set $S$ equals $e(\cZ)$, where $\cZ$ is a path with $n$ elements, called a {\em ribbon}, with down steps indexed by $S$.  Ribbons are also examples of posets arising from skew Young diagrams. Counting permutations with a given descent set is an important problem in combinatorics with a rich history including work by MacMahon \cite{CA}, Foulkes \cite{foulkes}, and Gessel--Reutenauer \cite{gessel1993counting}, and more recently by  Diaz-Lopez, Harris, Insko, Omar, and Sagan \cite{diaz2019descent}. Either by an inclusion-exclusion argument or from the Jacobi--Trudi identity for linear extensions of skew Young diagrams, there is a determinantal formula
\begin{equation} \label{eq:introdetdes}
e(\cZ) = \#\{w \in \mathfrak{S}_n \mid \Des(w) = S\} = n! \det\left(\frac{1}{(s_{j+1}-s_i)!}\right)_{0\leq i,j\leq k},
\end{equation}
where $s_0=0$ and $s_{k+1}=n$. 

In addition, more recently, linear extensions of tree posets have appeared in the context of noncrossing partitions and quiver representation theory \cite{garvermatherne}, as well as in pattern avoidance \cite{dwyerelizalde}.

The main result of this paper is to give a determinantal formula for the number of linear extensions of {\em mobile posets}, a class of posets which includes both ribbons and $d$-complete posets.

A mobile poset\footnote{The name ``mobile'' was chosen for the poset's resemblance to mobiles for babies and to the kinetic sculptures of Alexander Calder.} $\cP$ is a poset obtained from a ribbon poset $\cZ$ by allowing every element $z$ in $\cZ$ to cover the maximal element of a nonnegative number of disjoint $d$-complete posets, and by letting at most one element $z'$ of $\cZ$ be covered by a certain element of a $d$-complete poset (see Figure~\ref{fig:mobile tree poset}: Left). If the $d$-complete posets in this description are restricted to rooted tree posets, then the posets in the resulting family are called {\em mobile tree posets} (see Figure~\ref{fig:mobile tree poset}: Right). 

The determinantal formula for mobile posets below is a common refinement to the formulas for $e(\cP)$ of Proctor  and MacMahon for $d$-complete posets and ribbons, respectively.

\begin{TT1} 
Let $\cP$ be a mobile poset with $n$ elements and $F$ be the set of path folds for $\cP$ (see \eqref{eq: folds mobile tree poset}). Then 
\begin{equation}
e(\cP) \,=\, n!\cdot\det(M_{i,j})_{0 \leq i,j\leq k}, \quad \text{for} \quad 
M_{i,j} := \begin{cases}0 & \text{if } j<i-1,\\
1 & \text{if } j = i-1,\\
1/\prod_{x \in \cP_{i,j}} h_{\cP_{i,j}}(x) & \text{otherwise},
\end{cases}
\end{equation}
where $k$ is the size of $F$ and $\cP_{i,j}$ are certain $d$-complete posets (see Definition \ref{def: component array}).
\end{TT1}

\begin{figure}
\begin{center}
     \includegraphics[scale=0.7]{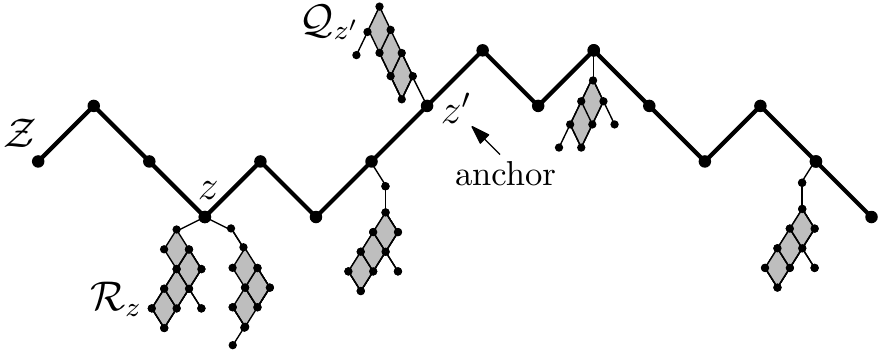}
     \qquad
     \includegraphics[scale=0.7]{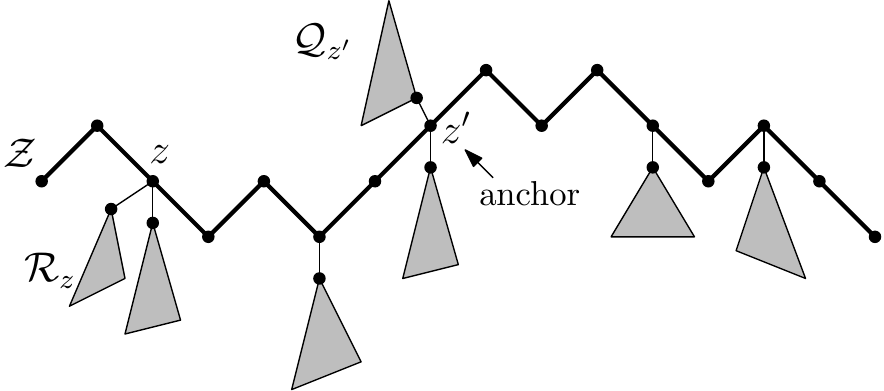}
\end{center}
\caption{Left: schematic of a mobile poset. Right: schematic of a mobile tree poset. The shaded rhombi depict $d$-complete posets, and the shaded triangles depict rooted tree posets.}
\label{fig:mobile tree poset}
\end{figure}

\subsection{\texorpdfstring{$q$}{q}-analogues}

Two well-studied $q$-analogues of the number of linear extensions for a labeled poset $(\cP,\omega)$ are the major index and inversion $q$-analogues:
\[
e_q^{\maj}(\cP,\omega):=\sum_{\sigma} q^{\maj(\sigma)} \qquad \text{and} \qquad e_q^{\inv}(\cP,\omega):=\sum_{\sigma} q^{\inv(\sigma)},
\]
where both sums are over linear extensions $\sigma$ of $(\cP,\omega)$ (see Section~\ref{section:qanalogues}), and $\maj(\sigma)$ and $\inv(\sigma)$ are the major index and number of inversions of $\sigma$, respectively. 
Bj\"orner--Wachs \cite{bw91} gave both major index and inversion $q$-analogues of \eqref{eqn:introprod} for rooted trees, and Stanley \cite{St71p2} gave a major index $q$-analogue of \eqref{eqn:introprod} for posets from Young diagrams of partitions that was generalized to $d$-complete posets by Proctor \cites{ProClass,p99}. For ribbon posets $\cZ$, a major index $q$-analogue of \eqref{eq:introdetdes} comes from the Jacobi--Trudi formula, and an inversion $q$-analogue is due to Stanley \cite[Corollary 3.2]{Stanqdes}. 

We give a $q$-analogue of Theorem~\ref{thm:gmobiledetdone} for mobile posets that is a common refinement to the respective major index $q$-analogues for $d$-complete posets and ribbons.

\begin{TT2}
Let $(\cP,\omega)$ be a labeled mobile poset with $n$ elements, and let $F$ be the set of path folds for $\cP$ (see \eqref{eq: folds mobile tree poset}). Then
\begin{equation}
e_q^{\maj}(\cP,\omega)
 \,=\, [n]_q!\cdot\det(M_{i,j})_{0 \leq i,j\leq k}, \quad \text{for} \quad 
M_{i,j} := \begin{cases}0 & \text{if } j<i-1,\\
1 & \text{if } j = i-1,\\
\frac{q^{\maj(\cP_{i,j},\omega_{i,j})}}{\prod_{x \in \cP_{i,j}} [h_{\cP_{i,j}}(x)]_q} & \text{otherwise},
\end{cases}
\end{equation}
where $k$ is the size of $F$ and $(\cP_{i,j},\omega_{i,j})$ are certain labeled $d$-complete posets (see Definition~\ref{def: component array}). 
\end{TT2}

We also give a determinantal formula for the inversion $q$-analogue of the number of linear extensions for mobile tree posets (Theorem~\ref{thm:qmobiledetdone}). This result is a common refinement of results of Bj\"orner--Wachs and Stanley on rooted tree posets and ribbons. 

To obtain our determinantal formulas, we carefully study the effect that folding cover relations in $\cP$ has on $e(\cP)$ via inclusion-exclusion and use a lemma by Stanley \cite[Example 2.2.4]{EC1} to turn certain inclusion-exclusion expressions into determinants. We are able to obtain hook-length formulas for the entries of the determinant by using properties of $d$-complete posets developed by Proctor in \cite{ProClass}. 

\subsection*{Outline}

The rest of the paper is organized as follows. In Section~\ref{section:prelim} we give definitions, notation, and background results on linear extensions, the previously known poset families mentioned above, and $q$-analogues of $e(\cP)$. Section~\ref{sec:folding} gives the inclusion-exclusion approach to compute $e(\cP)$ by folding certain cover relations of the posets and determines when the resulting formulas can be written as determinants. In
Section~\ref{section:mobile} we recover the determinant formula \eqref{eq:introdetdes} for $e(\cZ)$ using this approach, and we introduce the family of mobile posets and prove the main result (Theorem~\ref{thm:gmobiledetdone}). This section also gives applications of our main result to the simpler class of mobile tree posets. These posets are exactly the tree posets where our methods apply (see Theorem~\ref{cor:mobile characterization} for a precise statement). Section~\ref{ex:eulerext} has two infinite families of mobile tree posets whose number of linear extensions generalize {\em Euler numbers}. In Section~\ref{section:q} we prove the $q$-analogues of our main result (Theorems~\ref{thm:qmajgmobiledetdone} and \ref{thm:qmobiledetdone}). We conclude in Section~\ref{sec:final_rem} with results and open questions related to this work on (i) positive formulas for the number of linear extensions of mobile posets related to the Naruse--Okada formula \cite{NO} for {\em skew $d$-complete posets}, (ii) polynomiality for the number of linear extensions of mobile tree posets, and (iii) a $q$-analogue of Atkinson's algorithm (Lemma~\ref{prop:qatk}) for the number of linear extensions of an arbitrary tree poset.

\subsection*{Acknowledgments}

We are grateful to Igor Pak and Bruce Sagan whose comments and questions motivated Sections~\ref{ex:eulerext} and \ref{subsec:desc_poly}, respectively. We are also thankful to David Barrington, Sergi Elizalde, Jang Soo Kim, Peter Winkler, and Benjamin Young for helpful discussions and suggestions that led to improvements in this paper.

\section{Preliminaries}
\label{section:prelim}

For nonnegative integers $i\leq m$, we use $[i,m]$ to denote the set $\{i,i+1,\ldots,m\}$ and $[m]$ to denote the set $[1,m]$. 

\subsection{Posets and linear extensions}
\label{section:posets}

A {\em partially-ordered set} (poset) is a pair $(\cP, \le_\cP)$, with $\cP$ a finite set and $\le_\cP$ a binary relation on $\cP$ that is reflexive, antisymmetric, and transitive. We denote a poset by its underlying set when the order relation is clear from context. Throughout, we view $\le_\cP$ as both a subset of $\cP^2$ and as a way to compare two elements of $\cP$, depending on context.  (Thus, writing $(x,y) \in\ \le_{\cP}$ is equivalent to writing $x \le_{\cP} y$.)  We denote the set of cover relations of $\cP$ by $\lessdot_\cP$. An {\em (induced) subposet} $\cQ$ of $\cP$ is a poset whose underlying set is a subset of the elements of $\cP$, and whose relations are given by $s \leq_{\cQ} t$ if and only if $s \leq_{\cP} t$.  Given two elements $x,y \in \cP$, the {\em interval} $[x,y]$ is the subposet $\{z \in \cP \mid x\leq_{\cP} z \leq_{\cP} y\}$.

Throughout, we represent posets by their {\em Hasse diagrams}, and we say $\cP$ is {\em connected} if its Hasse diagram is connected. A {\em chain} (resp. {\em antichain}) is a poset where any two elements are comparable (resp. {\em incomparable}).

If $\cP$ and $\cQ$ are two posets,
we define their {\em disjoint sum} $\cP + \cQ$ as the poset with underlying set the disjoint union $\cP \sqcup \cQ$ and with relations the disjoint union $\le_\cP \sqcup \le_\cQ$.  If $E \subset \cP$, we denote by $\cP \setminus E$ the poset with underlying set $\cP \setminus E$ and with relations $\le_{\cP \setminus E}\ :=\ \le_\cP \setminus\ \{(x,y) \in\ \le_\cP\ \mid x \in E \text{ or } y \in E\}$.
Given a poset $\cP$ with two incomparable elements $x$ and $y$, let $\cP \oplus \{(x,y)\}$ be the poset obtained by adding the cover relation $(x,y)$ and taking the transitive closure. 

\begin{definition}\label{def:slantsum}
Let $\cP$ and $\cQ$ be disjoint posets, let $p$ be an element in $\cP$, and let $q$ be an element in $\cQ$. The {\em slant sum} of $\cP$ and $\cQ$ at $p$ and $q$ is the poset $\cP {}^p\backslash_q \cQ := (\cP + \cQ) \oplus \{(q,p)\}$.
\end{definition}

\begin{definition}\label{def:iterated_slantsum}
Let $\cP,\cQ_1,\dots,\cQ_m$ be disjoint posets, let $p$ be an element in $\cP$, and let $q_i$ be an element in $\cQ_i$ for $i=1,\dots,m$. The {\em iterated slant sum} of $\cP,\cQ_1,\dots,\cQ_m$ at $p$ and $q_1,\dots,q_m$ is the poset $\cP \underset{i=1,\ldots,m}{\prescript{p}{}{\big{\backslash}_{q_i}}} \cQ_i := \left(\cP + \cQ_1 +\cdots + \cQ_m\right) \oplus \{(q_1,p),\dots,(q_m,p)\}$.
\end{definition}

We now introduce the main object of study in this paper: linear extensions of posets.

\begin{definition}
A {\em linear extension} of an $n$-element poset $\cP$ is a bijection $f\co \cP \to [n]$ that is order-preserving; that is, if $x \le_\cP y$, then $f(x) \le f(y)$.  We denote by $\cL(\cP)$ the set of all linear extensions of $\cP$, by $e(\cP):=\#\cL(\cP)$ the number of linear extensions of $\cP$, and by $\overline{e}(\cP):= e(\cP)/n!$ the probability that a permutation of the elements of $\cP$ is a linear extension of $\cP$.  
\end{definition}

The following standard result gives a formula for the number of linear extensions of a disjoint sum in terms of the number of linear extensions for the summands.

\begin{proposition}[{\cite[Example 3.5.4]{EC1}}]\label{prop:disjointsum}
The number of linear extensions of a disjoint sum of posets $\cP_i$, each with $n_i$ elements, is
\[
e(\cP_1 + \cdots + \cP_k) = \binom{n_1 + \cdots + n_k}{n_1,\ldots,n_k}e(\cP_1)\cdots e(\cP_k).
\]
\end{proposition}

\subsection{Families of posets with closed formulas for \texorpdfstring{$e(\cP)$}{e(P)}}
\label{section:closedformulas}

We recall some classes of posets for which the number of linear extensions has a closed formula. 

\subsubsection{Rooted trees}
\label{section:rootedtrees}

A {\em tree poset} is a poset with a connected acyclic Hasse diagram, and a {\em rooted tree poset} is a tree poset with exactly one maximal element (the {\em root}).  Rooted tree posets have a particularly nice product formula, due to Knuth, for the number of their linear extensions.

\begin{theorem}[{Knuth \cite[Theorem 5.1.4H]{knuth_1998}}]\label{prop:rootedtree}
The number of linear extensions of a rooted tree poset $\cP$ with $n$ elements is
\[
e(\cP) = \frac{n!}{\prod_{x \in \cP}h_{\cP}(x)},
\]
where $h_{\cP}(x) := \#\{y \in \cP \mid y \le_{\cP} x\}$ is the {\em hook length} of $x$ in $\cP$.
\end{theorem}

\subsubsection{Ribbons}
\label{section:ribbons}

Let $S = \{s_1,\ldots,s_k\}_{<} \subset [n-1]$, where the subscript $<$ indicates that $s_1<\cdots < s_k$.  A {\em ribbon poset} $\cZ^{(n)}_S$ with descent set $S$ is the tree poset with underlying set $\{z_1,\ldots,z_n\}$ whose cover relations are  $z_{i+1} \lessdot z_i$ if $i \in S$ and $z_i \lessdot z_{i+1}$ if $i \not\in S$.  These posets coincide with the cell posets of ribbon skew shapes, and their linear extensions naturally correspond to permutations of size $n$ with descent set $S$. Ribbon posets have a determinantal formula for their number of linear extensions due to MacMahon \cite[1.55, vol 1]{CA}, and this identity is a special case of the Jacobi--Trudi determinantal identity counting standard Young tableaux of a skew shape (see \cite[Section 7.16]{EC2}).

\begin{theorem}[{MacMahon \cite[1.55, vol 1]{CA}}]\label{prop:ribbon}
The number of linear extensions of a ribbon poset $\cZ^{(n)}_S$ is given by
\begin{equation}\label{eq:Detribbon}
e(\cZ^{(n)}_S) = n!\cdot \det\left( \frac{1}{(s_{j+1}-s_i)!} \right)_{0\leq i,j\leq k},
\end{equation}
where $s_0=0$ and $s_{k+1} = n$.
\end{theorem}

This result can be proved using the following lemma, due to Stanley, that turns an alternating sum over sets into a determinant. 

\begin{lemma}[{Stanley \cite[Example 2.2.4]{EC1}}] \label{lemma:det}
Let $g$ be any function defined on $[0,k+1]\times [0,k+1]$ that satisfies $g(i,i) = 1$ and $g(i,j)=0$ if $j<i$, and let $D_k$ be the sum
\[D_k := \sum_{1\le i_1< i_2 < \cdots < i_j \le k}  (-1)^{k-j} g(0,i_1)g(i_1,i_2)\cdots g(i_j, k+1).\]

Then $D_k = \det(E_k)$,
where $E_k=(e_{i,j})$ is the  $(k+1)\times (k+1)$ matrix with entries $e_{i,j} = g(i, j+1)$ for $(i,j)\in [0,k]\times [0,k]$.
\end{lemma}

We will use Lemma~\ref{lemma:det} to prove our main results in Sections~\ref{sec:folding}, \ref{section:mobile}, and~\ref{section:q}. 

\subsubsection{$d$-complete posets}
\label{section:d-complete}

Defined by Proctor in \cite{p99}, $d$-complete posets form a large class of posets containing rooted tree posets and posets arising from Young diagrams,\footnote{See \cite[Table 1]{ProClass} for a complete classification of $d$-complete posets.} while still retaining a hook-length formula for their number of linear extensions.  We recall their definition below (see Definition~\ref{def:dcomp}).

A poset $\cP$ has a {\em diamond} if there are four elements $w,x,y,z$ in $\cP$ such that $z$ covers $x$ and $y$, while $x$ and $y$ cover $w$. For $k \geq 3$, a {\em double-tailed diamond poset} $d_k(1)$ is a poset obtained by adding a $k-3$ chain to the top and  bottom of a diamond. The {\em neck} elements are the $k-2$ elements above the two incomparable elements of the diamond. A {\em $d_k$-interval} is an interval $[w,z]$ which is isomorphic to $d_k(1)$.   Additionally, for $k\geq 3$, a {\em $d_k^-$-convex set} is a $d_k$-interval with the maximal element removed. Note that for $k\geq 4$, a $d_k^-$-convex set is an interval.

\begin{definition}[\cites{ProClass,p99,Okada}]\label{def:dcomp}
A poset $\cP$ is \emph{$d$-complete} if, for any $k\geq 3$, the following properties are satisfied:

\begin{enumerate}
    \item If $I$ is a $d_k^-$-convex set, then there exists an element $p$ such that $p$ covers the maximal elements of $I$, and $I\cup p$ is a $d_k$-interval.
    \item  If $I=[w, z]$ is a $d_k$-interval, then $z$ does not cover any elements of $\cP$ outside $I$.
    \item There are no $d_k^-$-convex sets which differ only in their minimal elements.
\end{enumerate}
\end{definition}

A connected $d$-complete poset has a unique maximal element \cite[Section 14]{ProClass}.  Given a connected $d$-complete poset $\cP$, its {\em top tree} $\Gamma$ is the subgraph of the Hasse diagram of $\cP$ consisting of vertices $x$ in $\cP$ such that $y \geq_{\cP} x$ is covered by at most one other element. (This subgraph is indeed a tree.) An element $y$ of $\cP$ is {\em acyclic} if $y \in \Gamma$ and is not part of the neck of any $d_k$-interval of $\cP$. Note that if $\cP$ is a rooted tree, then $\Gamma = \cP$ and all its elements are acyclic.

Slant sums (see Definition~\ref{def:slantsum}) can be used to combine two $d$-complete posets to obtain a larger $d$-complete poset.

\begin{proposition}[{Proctor \cite[Proposition B]{ProClass}}] \label{prop: slant sum tree d complete}
Let $\cP_1$ be a connected $d$-complete poset with an acyclic element $y$, and  let $\cP_2$ be a connected $d$-complete poset with maximal element $x$. Then the slant sum $\cP:=\cP_1 {}^y\backslash_x \cP_2$ is a $d$-complete poset, and the acyclic elements of $\cP_1$ and $\cP_2$ are acyclic elements of $\cP$.
\end{proposition}

Next, we recall the hook-length formula for the number of linear extensions of a $d$-complete poset. 

\begin{definition}[\cites{p99,ProClass,Okada}]\label{def:dcomphook}
The \emph{hook length} $h_{\cP}(z)$ of an element $z$ in a $d$-complete poset $\cP$ is defined as follows:
\begin{enumerate}

    \item If $z$ is not in the neck of any $d_k$-interval, then $h_{\cP}(z) = \#\{y \mid y \le_{\cP} z\}$.
    \item If $z$ is in the neck of a $d_k$-interval, then we can find some element $w$ such that $[w, z]$ is a $d_{\ell}$-interval for some $\ell$. If $x$ and $y$ are the two incomparable elements in the $d_{\ell}$-interval, then $h_{\cP}(z) = h_{\cP}(x) + h_{\cP}(y) - h_{\cP}(w)$. 
\end{enumerate}
\end{definition}

\begin{theorem}[{Peterson--Proctor \cite{ProHLF}}]\label{thm:dhook}
The number of linear extensions of a $d$-complete poset $\cP$ with $n$ elements is
\[
e(\cP) = \frac{n!}{\prod_{x \in \cP} h_\cP(x)},
\]
where $h_{\cP}(x)$ is the hook length of $x$ in $\cP$ from Definition~\ref{def:dcomphook}.
\end{theorem}

\subsection{\texorpdfstring{$q$}{q}-analogues of the number of linear extensions}
\label{section:qanalogues}

We recall $q$-analogues of Theorems~\ref{prop:rootedtree}, \ref{prop:ribbon}, and \ref{thm:dhook} and prepare the preliminaries for the $q$-analogues of our results, which we display in Section~\ref{section:q}.  

A \emph{labeled poset} $(\cP, \omega)$ is a poset $\cP$ with $n$ elements, together with a bijection $\omega\co \cP \to [n]$.  We call $\omega$ a \emph{natural labeling} if for any $x,y\in \cP$ with $x <_{\cP} y$, we have $\omega(x) < \omega(y)$. A labeling $\omega$ is \emph{regular} if we have the following: for all $x <_{\cP} z$ and $y\in \cP$, if $\omega(x) < \omega(y) < \omega(z)$ or $\omega(x) > \omega(y) > \omega(z)$ then $x <_{\cP} y$ or $y <_{\cP} z$.  For more on natural and regular labelings, we point to \cite[Chapter 3]{EC1} and \cite{q_hook} respectively.

\begin{definition}
Let $(\cP,\omega)$ be a labeled poset.  If $f\co \cP \rightarrow [n]$ is a linear extension of $\cP$, then the permutation $\omega \circ f^{-1}\in \mathfrak{S}_n$ is called a \emph{linear extension of the labeled poset} $(\cP,\omega)$.  We write $\cL(\cP,\omega)$ for the set of all linear extensions of $(\cP,\omega)$.
\end{definition}

\begin{definition}\label{def:invmaj}
Let $\sigma = \sigma_1\cdots \sigma_n \in \mathfrak{S}_n$ be a permutation, and define $\Des(\sigma) := \{i \in [n-1] \mid \sigma_i > \sigma_{i+1}\}$.  The \emph{inversion index} $\inv(\sigma)$ of $\sigma$ is the number of inversions of $\sigma$, and the \emph{major index} of $\sigma$ is
\[
\maj(\sigma) = \sum_{i \in \Des(\sigma)} i.
\]
\end{definition}

The \emph{descent set} of a labeled poset $(\cP,\omega)$ is given by $\Des(\cP,\omega) = \{x \in \cP \mid x \lessdot_\cP y \text{ and } \omega(x) > \omega(y)\}$. There is a version of the inversion index in Definition~\ref{def:invmaj} for labeled posets and of the major index for labeled $d$-complete posets.

\begin{definition}
The \emph{inversion index} of a labeled poset $(\cP,\omega)$ is the number of inversions of $(\cP,\omega)$:
\[
\inv(\cP,\omega) = \#\{(x,y) \mid \omega(y) < \omega(x) \text{ and } x <_{\cP} y \}.
\]
\end{definition}

\begin{definition}\label{def: major index d-complete}
The {\em major index} of a labeled $d$-complete poset $(\cP,\omega)$ is the sum of the hook lengths at the descents of $\omega$:
\[
\maj(\cP,\omega) = \sum_{x \in \Des(\cP,\omega)} h_{\cP}(x).
\]
\end{definition}

We are now ready to state the $q$-analogues of the theorems from the previous section.  For each integer $n\ge 1$, we define the {\em $q$-integer} $[n]_q := 1 + q + q^2 + \cdots + q^{n-2} + q^{n-1}$ and the {\em $q$-factorial} $[n]_q! := [n]_q [n-1]_{q} \cdots [2]_q [1]_q$. We also define the {\em $q$-multinomial coefficient} $\txtqbin{m}{\ell_1,\ldots,\ell_k}{q} := \frac{[m]_q!}{[\ell_1]_q! \cdots [\ell_k]_q!}$.

\begin{definition}\label{def:maj-inv-qanalog}
 The \emph{major index (inversion) $q$-analogue} of the number of linear extensions of a labeled poset $(\cP,\omega)$ with $n$ elements  is
\[
e^{\stat}_q(\cP, \omega) :=  \sum_{\sigma \in \cL(\cP,\omega)} q^{\stat(\sigma)},
\]
where $\stat \in \{\maj,\inv\}$. We write $\overline{e}^{\stat}_q(\cP, \omega) := e^{\stat}_q(\cP, \omega)/[n]_q!$ for the major index (inversion) $q$-analogue of $\overline{e}(\cP)$.
\end{definition}
 
Note that setting $q = 1$ in Definition~\ref{def:maj-inv-qanalog} recovers $e(\cP)$ and $\overline{e}(\cP)$. Moreover, for the major index, Stanley showed that the rational function $e_q^{\maj}(\cP,\omega)(1-q)^n/[n]_q!$ gives the generating function of {\em $(\cP,\omega)$-partitions} \cite[Theorem 3.15.7]{EC1}.

The next standard results are $q$-analogues of Proposition~\ref{prop:disjointsum}.

\begin{proposition}[{\cite[Exercise 3.162(a)]{EC1}}] \label{prop:qdjsum-maj}
Let $(\cP_1+\cdots + \cP_k,\omega)$ be a labeled poset that is the disjoint sum of posets $\cP_i$, each with $n_i$ elements. Then
\[
e_q^\maj((\cP_1,\omega_1) + \cdots + (\cP_k,\omega_k)) = \qbin{n_1 + \cdots + n_k}{n_1,\ldots,n_k}{q} e_q^\maj(\cP_1,\omega_1)\cdots e_q^\maj(\cP_k,\omega_k),
\]
where $\omega_i$ is the labeling obtained by restricting $\omega$ to $\cP_i$.
\end{proposition}

\begin{proposition}[{\cite[Section 3]{q_hook}}] \label{prop:qdjsum-inv}
Let $(\cP_1+\cdots + \cP_k,\omega)$ be a labeled poset that is the disjoint sum of posets $\cP_i$, each with $n_i$ elements. Suppose that 
$\omega$ has the property that every element of the label set $\omega(\cP_i)$ is smaller than every element of the label set $\omega(\cP_j)$ whenever $i<j$. Then
\[
e_q^\inv((\cP_1,\omega_1) + \cdots + (\cP_k,\omega_k)) = \qbin{n_1 + \cdots + n_k}{n_1,\ldots,n_k}{q} e_q^\inv(\cP_1,\omega_1)\cdots e_q^\inv(\cP_k,\omega_k),
\]
where $\omega_i$ is the labeling obtained by restricting $\omega$ to $\cP_i$.
\end{proposition}

We have the following $q$-analogues, by Bj\"{o}rner and Wachs \cite{q_hook}, of the Knuth hook-length formula in Theorem~\ref{prop:rootedtree}.

\begin{theorem}[{Bj\"orner--Wachs \cite[Theorems 1.1 and 1.2]{q_hook}}] \label{lemma:qhook}
    Let $(\cP,\omega)$ be a labeled rooted tree poset with $n$ elements, where $\omega$ is a regular labeling. Then
    \begin{equation}\label{eq: inv q-analogue rooted trees}
        e_q^{\inv}(\cP,\omega) = q^{\inv(\cP,\,\omega)} \frac{[n]_q!}{\prod_{x\in \cP}[h_{\cP}(x)]_q}.
 \end{equation}
    Moreover, if $\omega$ is any labeling, then
    \begin{equation} \label{eq: maj q-analogue rooted trees} 
e_q^{\maj}(\cP,\omega)
 = q^{\maj(\cP,\,\omega)} \frac{[n]_q!}{\prod_{x\in \cP}[h_{\cP}(x)]_q}.
    \end{equation}
\end{theorem}
Note that if $\omega$ were a natural labeling, then $\inv(\cP,\omega) = 0$ in the first formula, which removes a power of $q$. 

We also have a $q$-analogue of the ribbon poset formula in Theorem~\ref{prop:ribbon}, due to Stanley.

\begin{proposition}[{Stanley \cite[Example 2.2.5]{EC1}, \cite[Corollary 3.2]{Stanqdes}}]\label{prop:stanleyqribbon}
Let $(\cZ_S^{(n)},\omega)$ be a labeled poset with $\cZ_S^{(n)}$ a ribbon poset and $\omega$ a natural labeling. Then
\[
e_q^{\inv}(\cZ^{(n)}_S,\,\omega) 
 \, = \, [n]_q! \cdot \det\left(\frac{1}{[s_{j+1}-s_i]_q!}\right)_{0\leq i,j\leq k},
\]
where $s_0=0$ and $s_{k+1} = n$.
\end{proposition}

Lastly, Peterson and Proctor showed the following $q$-analogue of Theorem~\ref{thm:dhook} and generalization of \eqref{eq: maj q-analogue rooted trees}. See \cite{d_comp} for a recent proof by Kim and Yoo using {\em $q$-integrals}. 
 
\begin{theorem}[{Peterson--Proctor \cite{ProHLF}}] \label{lemma:qhookmaj d complete}
    Let $(\cP,\omega)$ be a labeled poset with $\cP$ a $d$-complete poset having $n$ elements.  Then
    \[e_q^{\maj}(\cP,\omega) = q^{\maj(\cP,\,\omega)} \frac{[n]_q!}{\prod_{x\in \cP}[h_\cP(x)]_q}.\]
\end{theorem}

\begin{remark}\label{rem:noinv}
 After searching the literature, the authors were unable to find an analogous result for $e_q^{\inv}(\cP,\omega)$ when $\cP$ is a $d$-complete poset.
\end{remark}

\section{Inclusion-exclusion and determinant formulas for counting linear extensions}\label{sec:folding}

\subsection{Folding and an alternating formula for linear extensions}

We begin with a simple inclusion-exclusion formula for $e(\cP)$.

\begin{definition}
Let $\cP$ be a poset, $F \subset\ \lessdot_\cP$, and $F^{\op} := \{(y,x) \in \cP^2 \mid (x,y) \in F\}$. We write $\cP \ominus F$ for the poset with the same underlying set as $\cP$, but with cover relations $\lessdot_{\cP\setminus F}\ :=\ \lessdot_\cP\!\! \setminus F$.  We call a {\em fold} of $\cP$ at $F$ the poset
\[
\cP_F := (\cP \ominus F) \oplus F^{\op}
\]
obtained by deleting the cover relations in $F$, adding the opposite cover relations, and taking the transitive closure. If $S\subset F$, then we call a {\em partial fold} of $\cP$ at $S$ the poset
\[
\cP_{S,F} := (\cP \ominus F) \oplus S^{\op}.
\]
\end{definition}

The next lemma uses inclusion-exclusion to describe how the number of linear extensions of a poset changes when folding at a single cover relation.

\begin{lemma}\label{lemma:inclusionexclusion}
Let $\cP$ be a poset and $(x,y)$ be in $\lessdot_{\cP}$. Then
\[
e(\cP) = e(\cP \ominus \{(x,y)\}) - e(\cP_{\{(x,y)\}}).
\]
\end{lemma}

\begin{proof}
The linear extensions of $\cP$ are also linear extensions of $\cP \ominus \{(x,y)\}$. The linear extensions of the latter that are not linear extensions of the former are exactly those extensions where $y$ appears before $x$. Such linear extensions are in bijection with the linear extensions of the poset $(\cP \ominus \{(x,y)\}) \oplus \{(y,x)\} = \cP_{\{(x,y)\}}$. Thus, as permutations of the same underlying set $\cP$, we have 
\begin{equation} \label{linext IE decomposition}
\cL(\cP) = \cL(\cP \ominus \{(x,y)\}) \,\setminus\, \cL(\cP_{\{(x,y)\}}).
\end{equation}
The result follows by calculating the cardinality of each of these sets.
\end{proof}

\begin{example} \label{ex:poset1fold}
Consider the seven element poset $\mathcal{P}$ in Figure~\ref{fig:ex-ie-posets}: Left.  Choosing either the cover relation $(c,e)$ or $(a,c)$, we obtain
\begin{align}
    77= e(\cP) &=  e(\cP \ominus \{(c,e)\}) - e(\cP_{\{(c,e)\}}) = 105-28, \label{eq:foldbridge}\\
    &= e(\cP \ominus \{(a,c)\}) - e(\cP_{\{(a,c)\}}) = 117-40. \label{eq:foldnonbridge}
\end{align}
See Figure~\ref{fig:ex-ie-posets}: Right for an illustration of these inclusion-exclusion formulas.
\renewcommand{\qedsymbol}{\rule{0.65em}{0.65em}}\hfill\qedsymbol
\begin{figure}
\begin{center}
    \raisebox{10mm}{\raisebox{3mm}{$\cP =$} \includegraphics[scale=0.6]{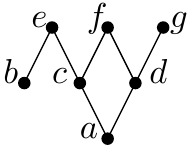}} \qquad \qquad 
   \includegraphics[scale=0.6]{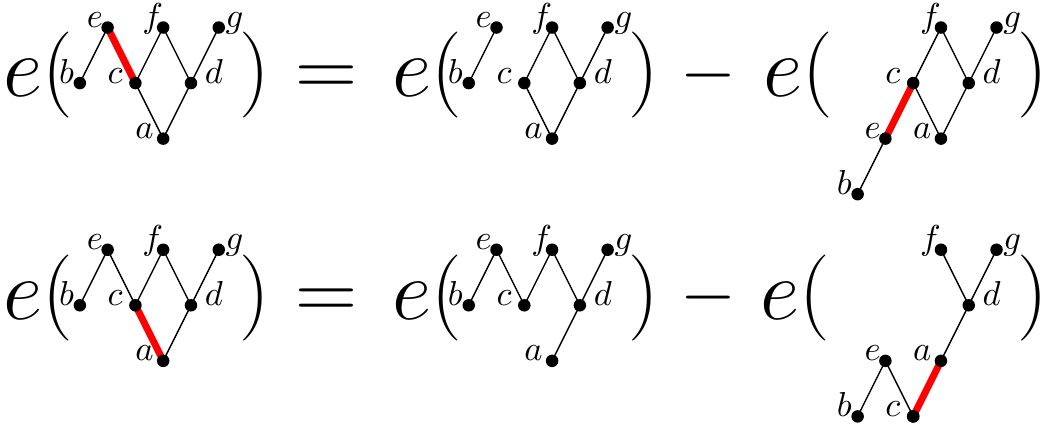}
\end{center}
\caption{Example of using the inclusion-exclusion formula from Lemma~\ref{lemma:inclusionexclusion} to calculate $e(\cP)$.}
\label{fig:ex-ie-posets}
\end{figure}
\end{example}

\begin{remark} \label{rem:delete-bridge-fold}
Note from \eqref{eq:foldbridge} that if the cover relation $(x,y)$ is a bridge (that is, it does not belong to a cycle in the Hasse diagram of $\cP$), then the Hasse diagram of the folded poset $\cP_{\{(x,y)\}}$ is obtained by simply folding the cover relation $(x,y)$ in the Hasse diagram of $\cP$. If the cover relation belongs to a cycle of the Hasse diagram of $\cP$, as in \eqref{eq:foldnonbridge}, then the change in the Hasse diagram when forming the folded poset $\cP_{\{(x,y)\}}$ is more significant. 
\end{remark}

\begin{corollary}\label{corollary:alternating}
Let $\cP$ be a poset, and let $F \subset\ \lessdot_\cP$.  Then
\begin{equation} \label{equation:alt}
e(\cP) = \sum_{S\subset F}(-1)^{\#S} e(\cP_{S,F}).
\end{equation}
\end{corollary}

\begin{proof}
The result follows from repeated application of Lemma~\ref{lemma:inclusionexclusion}.
\end{proof}

\begin{example}\label{ex:poset2folds}
 For the poset $\cP$ from Example~\ref{ex:poset1fold}, Corollary~\ref{corollary:alternating} yields the formula
 \begin{equation} \label{eq:example-inclusion-exclusion-2folds}
      \raisebox{-20pt}{\includegraphics[scale=0.6]{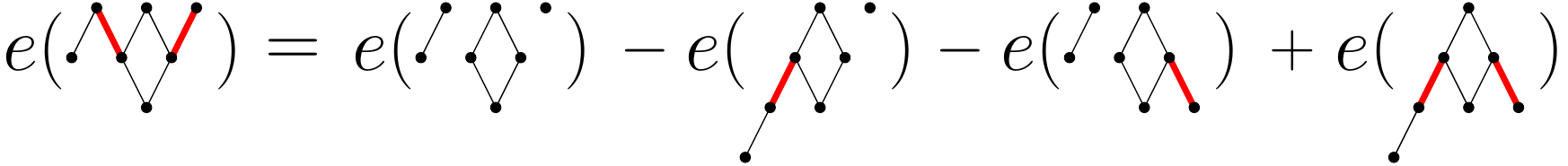}}  
\end{equation}
 when $F=\{(c,e),(d,g)\}$ is the set of cover relations depicted in red on the left-hand side of the above equation.\renewcommand{\qedsymbol}{\rule{0.65em}{0.65em}}\hfill\qedsymbol
\end{example}

\subsection{Component trees and component arrays}
\label{section:detie}

For the rest of the paper, we assume that $\cP$ is connected\footnote{Given a disconnected poset $\cP$, one can first apply the techniques of this paper to each connected component separately.  This will compute the number of linear extensions of each connected component.  Then, one can use Proposition~\ref{prop:disjointsum} to compute $e(\cP)$.}  and that the set of folds $F$ consists only of {\em bridges} in the Hasse diagram of $\cP$; that is, the cover relations that we fold do not lie in a cycle in the Hasse diagram of $\cP$ (see Remark~\ref{rem:delete-bridge-fold}).  In this case, the Hasse diagram of $\cP_{S,F}$ is obtained from the Hasse diagram of $\cP$ by folding the edges in $S$ and deleting the edges in $F\setminus S$. 

\begin{definition}
We define the {\em component tree} of $\cP_F$ to be the tree $C(\cP_F)$ with vertices the connected components of the poset $\cP \ominus F$, and with edges $(C_x, C_y)$ for each cover relation $(x,y)$ in $F$ where $x \in C_x$ and $y\in C_y$. 
\end{definition}

\begin{definition} \label{def: component array}
Suppose $\#F = k$ and $\sigma = (\sigma_0,\sigma_1,\ldots,\sigma_k)$ is a total order on the vertices of $C(\cP_F)$.  The \emph{component array} $M_\sigma(\cP_F)$ is the triangular array of posets
\[
(M_\sigma(\cP_F))_{i,j} := C(\cP_F)[i,j],
\]
where  $0\leq i\leq j \leq k$ and $C(\cP_F)[i,j]$ is the subposet of $\cP_F$ on the elements in the connected components $\sigma_i,\sigma_{i+1},\ldots, \sigma_j$ of $\cP \ominus F$.
\end{definition}

\begin{definition}\label{def:pathorder}
Let $C(\cP_F)$ be the component tree of a folded poset $\cP_F$.  A total order $\sigma = (\sigma_0,\sigma_1,\ldots,\sigma_k)$ on the vertices of $C(\cP_F)$ is called a \emph{path order}\footnote{The choice of this name will be justified in Proposition~\ref{prop:goodordering} and Remark~\ref{rem: path order}.} if each entry of the component array $M_\sigma(\cP_F)$ is a connected poset.
\end{definition}

\begin{example}\label{ex:nonzig3}
For the poset $\cP$ and folds $F$ from Example~\ref{ex:poset2folds}, Figure~\ref{fig: component tree and array} depicts the component tree and component arrays for a path order $\sigma=(\sigma_0,\sigma_1,\sigma_2)$ and an order $\tau=(\sigma_1,\sigma_0,\sigma_2)$ that is not a path order.\renewcommand{\qedsymbol}{\rule{0.65em}{0.65em}}\hfill\qedsymbol 

\begin{figure}
\begin{center}
\raisebox{22pt}{\includegraphics[scale=0.8]{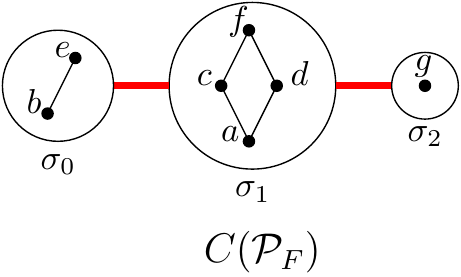}} \qquad \,\,  \includegraphics[scale=0.8]{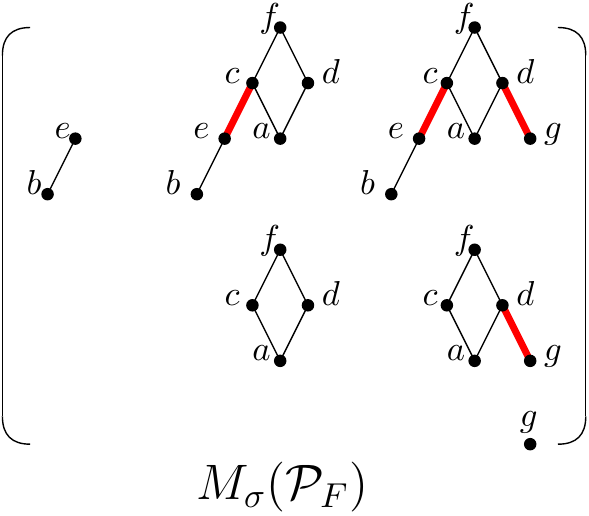} \qquad \,\, 
\includegraphics[scale=0.8]{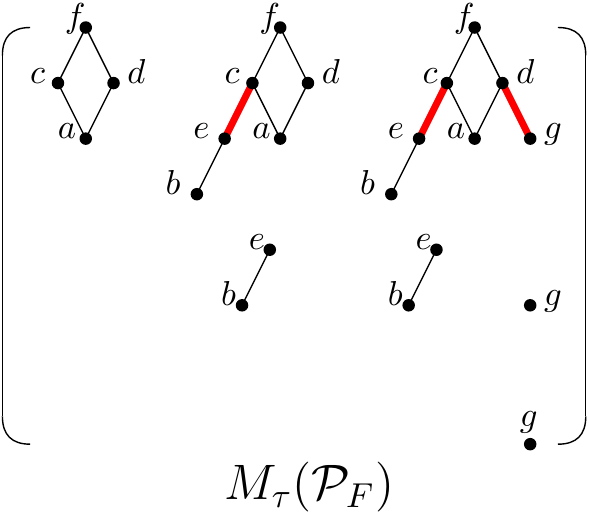}
\end{center}
\caption{Left: example of component tree $C(\cP_F)$ with a chosen total order $\sigma$ on the vertices. Center: example of component array $M_{\sigma}(\cP_F)$. Right: example of a component array $M_{\tau}(\cP_F)$ with a different order $\tau=(\sigma_1,\sigma_0,\sigma_2)$.}
\label{fig: component tree and array}
\end{figure}
\end{example}

\begin{proposition}\label{prop:goodordering}
There exists a path order $\sigma$ on the vertices of $C(\cP_F)$ if and only if $C(\cP_F)$ is a path.
\end{proposition}

\begin{proof}
We prove the ``if'' statement first, so suppose $C(\cP_F)$ is a path.  Thus, $C(\cP_F)$ has exactly two degree-one vertices.  Travel along the path from one degree-one vertex to the other, labeling the vertices along the way in the order $\sigma_0, \sigma_1, \ldots, \sigma_k$.  The total order $\sigma = (\sigma_0,\sigma_1,\ldots,\sigma_k)$ is a path order.

Now we prove the ``only if'' statement, so suppose $C(\cP_F)$ is not a path and let $\sigma$ be a total order on the vertices of $C(\cP_F)$.  Since $C(\cP_F)$ is a tree, it follows that there exists a vertex $X$ in $C(\cP_F)$ with degree at least three. Let $A,B,C$ be three of the vertices in $C(\cP_F)$ adjacent to $X$ labeled in the order they appear in $\sigma$; that is, there are indices $0\leq r<s<t\leq k$ such that $\sigma_r=A, \sigma_s=B$, and $\sigma_t=C$. Depending on where $X$ appears in $\sigma$, the poset $C(\cP_F)[r,s]$ has subposets $A$ and $B$ in different components or the poset  $C(\cP_F)[s,t]$ has subposets $B$ and $C$ in different components.  Thus, $\sigma$ is not a path order, as desired.
\end{proof}

\begin{remark} \label{rem: path order}
The result above shows that there are exactly two path orders on the vertices of $C(\cP_F)$ when it is a path: traveling from either end of the path to the other (and there are no path orders otherwise). Either of these two orders $\sigma$ gives a total order $\sigma' = (\sigma'_0,\sigma'_1,\ldots,\sigma'_{k-1})$ on $F$, where $\sigma'_i$ denotes the element of $F$ that is incident with both $\sigma_i$ and $\sigma_{i+1}$ in $C(\cP_F)$; we call such a $\sigma'$ a \emph{path order} on $F$.  Furthermore, we will often abuse notation and use $\sigma$ for both a path order on the vertices of $C(\cP_F)$ and for its corresponding path order on $F$.
\end{remark}

\subsection{Determinant formula}

Given a component array $M_\sigma(\cP_F)$ with $\sigma$ a path order on the vertices of $C(\cP_F)$, we define a matrix $\overline{e}(M_{\sigma}(\cP_F))$ by
\begin{equation} \label{def: eq matrix}
\overline{e}(M_{\sigma}(\cP_F))_{i,j} := \left\{\begin{array}{ll}0 & \text{if } j<i-1 ,\\
1 & \text{if } j = i-1,\\
\overline{e}((M_\sigma(\cP_F))_{i,j}) & \text{otherwise}, \end{array}\right.
\end{equation}
where $(i,j)\in [0,k]\times [0,k]$.  The upper-triangular entries of this matrix are the probabilities that a permutation of the elements of the poset $(M_\sigma(\cP_F))_{i,j}$ yields a linear extension. 

\begin{lemma}\label{thm:det}
Let $\cP$ be a poset with $n$ elements, and let $\sigma$ be a path order on the vertices of $C(\cP_F)$.  Then we have
\[
e(\cP) = n!\cdot\det(\overline{e}(M_{\sigma}(\cP_F))).
\]
\end{lemma}

\begin{proof}
We encode $\sigma$ by the function $\phi\co F \rightarrow [k], \sigma_i \mapsto i+1$.  Now use the inclusion-exclusion formula in \eqref{equation:alt} to write $e(\cP)$ as the alternating sum 
\begin{equation}
    \label{equation:alt applied to T}
e(\cP) = \sum_{S\subset F}(-1)^{\#S} e(\cP_{S,F}).
\end{equation}
Since $F$ is a set of folds and $C(\cP_F)$ is a path, the poset $\cP_{S,F}$ is the following disjoint sum of subposets of $\cP_F = C(\cP_F)[0,k]$:
\begin{equation}\label{eqn:djsum}
\cP_{S,F} = C(\cP_F)[0,i_1-1] + C(\cP_F)[i_1,i_2-1] + \cdots + C(\cP_F)[i_j,k],
\end{equation}
where $\{i_1,i_2,\ldots,i_j\}_{<}$ is the image $\phi(F\setminus S)$ ordered by $\sigma$. Then by Proposition~\ref{prop:disjointsum} and the fact that $\#\cP_{S,F} = \#\cP$, we have
\[
e(\cP_{S,F})\,=\, (\#\cP)! \cdot   \overline{e}(M_{\sigma}(\cP_F)_{0,i_1-1})\overline{e}(M_{\sigma}(\cP_F)_{i_1,i_2-1}) \cdots \overline{e}(M_{\sigma}(\cP_F)_{i_j,k}).
\]
Then  \eqref{equation:alt applied to T}  becomes
\[
e(\cP) \,=\, (\#\cP)!  \sum_{\{i_1,i_2,\ldots,i_j\}_{<} \,\subset\, [k]}(-1)^{k-j} \cdot \overline{e}(M_{\sigma}(\cP_F)_{0,i_1-1})\overline{e}(M_{\sigma}(\cP_F)_{i_1,i_2-1}) \cdots \overline{e}(M_{\sigma}(\cP_F)_{i_j,k}).
\]
The result then follows by using  Lemma~\ref{lemma:det} with the function $g(i,j+1) = \overline{e}(M_{\sigma}(\cP_F)_{i,j})$ for $0\leq i \leq j \leq k$, $g(i,i)=1$, and $g(i,j)=0$ for $j<i$.
\end{proof}

\begin{example}\label{ex:pathexplanation}
We continue with Example~\ref{ex:nonzig3}. Applying Lemma~\ref{thm:det}  to the component array $M_\sigma(\mathcal{P}_F)$ in Figure~\ref{fig: component tree and array}: gives
\[
e(\cP) = 7! \cdot \det(\overline{e}(M_{\sigma}(\cP_F))) = 7!\left( \raisebox{-15pt}{\includegraphics[scale=0.5]{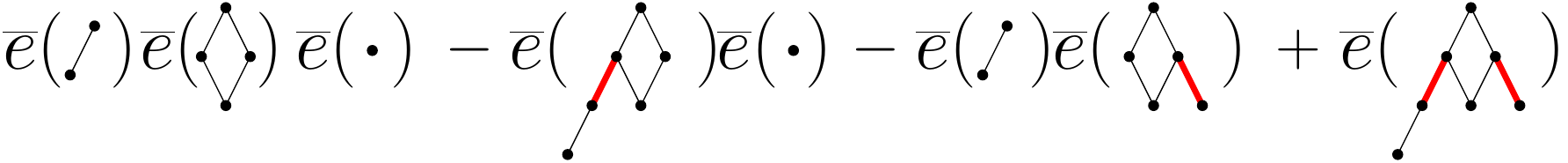}}\right),
\]
which evaluates to $77$ as expected (see Example~\ref{ex:poset1fold}).  The displayed expression matches the right-hand side of \eqref{eq:example-inclusion-exclusion-2folds} after applying Proposition~\ref{prop:disjointsum}.

Note that if instead we use the component array $M_{\tau}(\cP_F)$ from the order $\tau=(\sigma_1,\sigma_0,\sigma_2)$, which is not a path order (Figure~\ref{fig: component tree and array}: Right), then expanding $7!\cdot \det(\overline{e}(M_{\tau}(\cP_F)))$ does not match the right-hand side of \eqref{eq:example-inclusion-exclusion-2folds} nor give $e(\cP)$.\renewcommand{\qedsymbol}{\rule{0.65em}{0.65em}}\hfill\qedsymbol
\end{example}

\begin{remark}
We have chosen to work in the setting where the component array $M_\sigma(\cP_F)$ has connected poset entries (see Definition~\ref{def:pathorder}). As illustrated in Example~\ref{ex:pathexplanation} above, if we do not impose this restriction, then  Lemma~\ref{thm:det} does not necessarily hold. 
\end{remark}

\section{Determinant formulas for linear extensions of mobile posets}
\label{section:mobile}

In this section, we give the three main applications of Lemma~\ref{thm:det}.  As a warm-up, we recover the determinant formula for linear extensions of ribbons in Theorem~\ref{prop:ribbon}.  Then, we prove a determinant formula for linear extensions of {\em mobile posets} in Theorem~\ref{thm:gmobiledetdone} and use this result to prove Corollary~\ref{thm:mobiledetdone}, which gives a formula for the special case of {\em mobile tree posets}. 
 In these three formulas, the matrix entries are given by hook-length formulas of chains, of $d$-complete posets, and of rooted trees, respectively.

\subsection{Recovering determinant formula for ribbons}

The next result shows that MacMahon's determinant formula for the number of linear extensions of ribbon posets (or permutations with given descent set) can be recovered using Lemma~\ref{thm:det}.

\begin{proof}[Proof of Theorem~\ref{prop:ribbon}]
Given a ribbon poset $\cP := \cZ^{(n)}_S$ where $S=\{s_1<\cdots<s_k\}_{<} \subset [n-1]$, we choose the set of folds  $F = \{ (z_{i+1},z_i) \mid i \in S\}$.  The folded poset $\cP_F$ is a chain, so the component tree $C(\cP_F)$ is a path.  Now, Lemma~\ref{thm:det} applies to give
\[
e(\cP) = n! \cdot \det(\overline{e}(M_{\sigma}(\cP_F))),
\]
where we have chosen from the two possible path orders on the set $F$ the one that is consistent with the ordered set $S$. With this path order $\sigma$, each poset $C(\cP_F)[i,j]$ in the component array is a chain of size $s_{j+1}-s_i$, where $s_0=0$ and $s_{k+1}=n$. Since a chain has only one linear extension, it follows that $\overline{e}(C(\cP_F)[i,j]) = 1/(s_{j+1}-s_i)!$ and the desired formula follows.  
\end{proof}

\begin{example}
For the ribbon $\cP = \cZ^{(6)}_{\{3,5\}}$, we use the set of folds $F=\{(z_4,z_3), (z_6,z_5)\}$. The folded poset $\cP_F$ and the component array $M_{\sigma}(\cP_F)$ (where $\sigma$ is the path order on $F$ consistent with the ordered set $S$) are illustrated in Figure~\ref{fig:example ribbon}. Applying Lemma~\ref{thm:det} gives
\[
e(\cP) = 6! \cdot \det(\overline{e}(M_{\sigma}(\cP_F))) = 6! \cdot \det
\begin{pmatrix}
\frac{1}{3!} & \frac{1}{5!} & \frac{1}{6!} \\[6pt]
1 & \frac{1}{2!} & \frac{1}{3!} \\[6pt]
0 & 1 & \frac{1}{1!}
\end{pmatrix}
= 35,
\]
and this agrees with the formula in Theorem~\ref{prop:ribbon}. \renewcommand{\qedsymbol}{\rule{0.65em}{0.65em}}\hfill\qedsymbol
\begin{figure}
    \centering
    \includegraphics[scale=0.85]{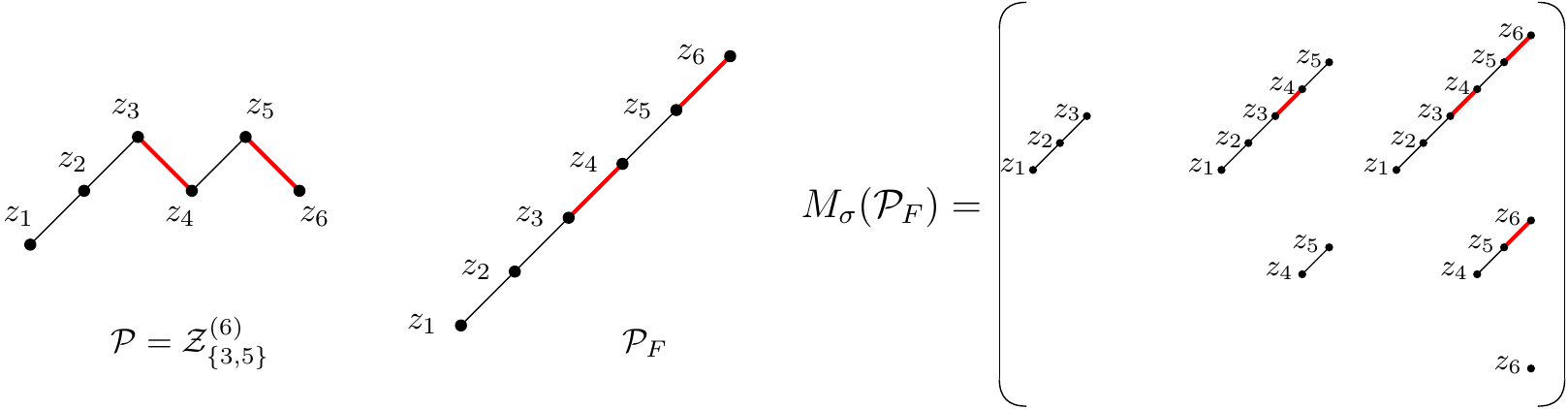}
    \caption{Left: a ribbon poset $\cP$. Center: the folded ribbon $\cP_F$. Right: its component array $M_\sigma(\cP_F)$.}
    \label{fig:example ribbon}
\end{figure}
\end{example}

\subsection{Determinant formula for mobile posets}
\label{section:new}

In this section, we use Lemma~\ref{thm:det} to give a determinant formula for the number of linear extensions of a class of posets we call ``mobile posets'', which generalize both rooted tree posets and ribbons. 

\begin{definition}\label{def:generalized mobiles}
A {\em mobile poset} $\cP$ is a poset obtained from a ribbon $\cZ$ by the following two operations:
\begin{compactitem}
\item[(i)] For every element $z \in \cZ$, perform an iterated slant sum $\cZ \underset{i=1,\ldots,m_z}{\prescript{z}{}{\big{\backslash}_{r_i}}} \cR_z^{(i)}$ with $m_z \geq 0$ connected $d$-complete posets $\cR_z^{(i)}$ with respective maximal elements $r_i$. Denote the resulting poset by $\cP'$.
\item[(ii)] For at most one element $z' \in \cZ$, perform a slant sum $\cQ_{z'} \,{}^{q}\backslash_{z'}\, \cP'$ where $\cQ_{z'}$ is a connected $d$-complete poset and $q$ is one of its acyclic elements. Such an element $z'$ is called an {\em anchor}.
\end{compactitem}
If no such element $z' \in \cZ$ is used in Operation (ii), we say that the mobile is {\em free-standing} with respect to the ribbon $\cZ$.
\end{definition}

See Figure~\ref{fig:mobile tree poset}: Left for a schematic of a mobile poset. 

\begin{example}
Figure~\ref{fig: example and nonexample of mobile} shows four examples: a free-standing mobile, a mobile, and two posets that cannot be expressed as mobiles. \renewcommand{\qedsymbol}{\rule{0.65em}{0.65em}}\hfill\qedsymbol

\begin{figure}
    \centering
    \includegraphics{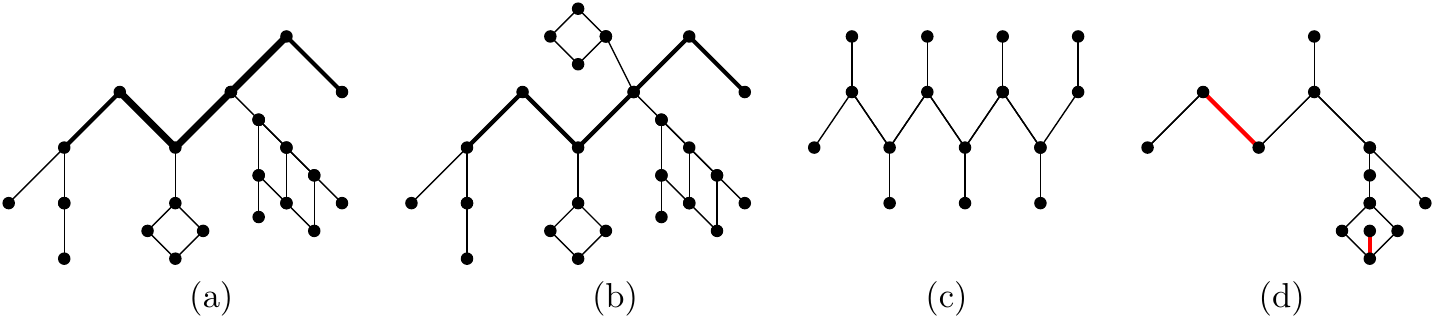}
    \caption{Examples of (a) a free-standing mobile poset, (b) a mobile poset, and (c) and (d) posets that are not mobile posets.}
    \label{fig: example and nonexample of mobile}
\end{figure}
\end{example}

\begin{remark} \label{rem:uniqueness ribbon}
 Note that, in a mobile poset there may be more than one choice of an underlying ribbon (see, for example, Figure~\ref{fig: example and nonexample of mobile}(a)). However, in a free-standing mobile poset, the possible underlying ribbons $\cZ$ that ensure the mobile is free-standing have a common unique subribbon $\cZ'$ (indicated in heavy bold in the figure) starting at the first descent of $\cZ$ and ending immediately after the last ascent of $\cZ$.
\end{remark}

Next, we show that certain mobile posets are also $d$-complete posets.

\begin{proposition} \label{prop:when-is-mobile-dcomplete}
Let $\cP$ be a mobile poset with underlying ribbon poset $\cZ$ where $\cZ$ has a unique maximal element. 
\begin{compactitem}
\item[(1)] If $\cP$ is free-standing with respect to $\cZ$ then $\cP$ is $d$-complete. 
\item[(2)] If $\cP$ is not free-standing with respect to $\cZ$ and the maximal element of $\cZ$ coincides with the anchor $z'$, then $\cP$ is $d$-complete.
\end{compactitem}
\end{proposition}

\begin{proof}
Since the ribbon $\cZ$ has a unique maximal element, it follows that $\cZ$ is a rooted tree. The mobile poset $\cP$ is obtained from the rooted tree $\cZ$ with elements $z_1,\ldots,z_{s}$ by doing the following:
\begin{compactitem}
\item[(a)] Let $\cT_0=\cZ$. For $\ell=1,\ldots,s$, let $\cT_{\ell} = \cT_{\ell-1} \underset{i=1,\ldots,m_{z_{\ell}}}{\prescript{z_{\ell}}{}{\big{\backslash}_{r_i}}} \cR_{z_{\ell}}^{(i)}$ for each of the $m_{z_\ell} \geq 0$ connected $d$-complete posets $\cR_{z_{\ell}}^{(i)}$ covered by $z_{\ell}$ in $\cZ$ in Operation (i) of Definition~\ref{def:generalized mobiles}. Denote the resulting poset $\cT_s$ by $\cP'$. 
\item[(b)] Perform the slant sum $\cQ_{z'} \,{}^{q}\backslash_{z'}\cP'$, where $q$ is an acyclic element of the $d$-complete poset $\cQ_{z'}$ from Operation (ii) of Definition~\ref{def:generalized mobiles}.
\end{compactitem}
Proposition~\ref{prop: slant sum tree d complete} guarantees that the poset $\cP'$ is a $d$-complete poset, and furthermore, that the poset $\cP =\cQ_{z'} \,{}^{q}\backslash_{z'}\cP'$ is a  $d$-complete poset, as desired.
\end{proof}

We will now produce a set of folds $F$ on any mobile poset $\cP$ that makes the component tree $C(\cP_F)$ a path.  After doing so, Proposition~\ref{prop:goodordering} will guarantee the existence of a path order on the vertices of $C(\cP_F)$, so we will be able to use Lemma~\ref{thm:det} to compute the number of linear extensions of $\cP$.

\begin{definition}
Let $\cP$ be a mobile poset and $\cZ := \cZ_S$ be an underlying ribbon poset of $\cP$.  The set  of \emph{path folds} for $\cP$ (with respect to $\cZ$) is defined as
\begin{equation}
    \label{eq: folds mobile tree poset}
F = \begin{cases}
\{ (z_{i+1},z_i) \mid i \in S\} &\text{ if } \cP \text{ is free-standing,}\\
\{ (z_{i+1},z_i) \mid i \in S, i<j\} \cup \{ (z_i,z_{i+1}) \mid i \not\in S, i \geq j\} & \text{ otherwise,}  
\end{cases}
\end{equation}
where $j$ is the index of the anchor $z'=z_j$ covered by an acyclic element of a connected $d$-complete poset $\cQ_{z'}$ (see Definition~\ref{def:generalized mobiles} (ii)). 
\end{definition}

\begin{lemma}\label{lemma: generalized mobiles have path folding}
Let $\cP$ be a mobile poset and $F$ be the set of path folds for $\cP$.  Then the component tree $C(\cP_F)$ is a path.
\end{lemma}

\begin{proof}
In the case of a free-standing mobile, the underlying folded ribbon $\cZ_F$ in the folded poset $\cP_F$ is a chain. Hence $\cP_F$ consists of an underlying chain $\cZ_F$, containing the folds, where for every element $z \in \cZ_F$ we do the corresponding slant sums in Operation~(i) of Definition~\ref{def:generalized mobiles} that we do for the mobile $\cP$.

If the mobile is not free-standing with respect to $\cZ$ and $z'=z_j$ is its anchor, then the underlying folded ribbon $\cZ_F$ in the folded poset $\cP_F$ is a rooted path rooted at $z'$. Hence $\cP_F$ consists of this underlying rooted path $\cZ_F$, containing the folds, where we do the corresponding slant sums in Operations~(i),(ii) of Definition~\ref{def:generalized mobiles} that we do for the mobile $\cP$.

In both cases, we have an underlying path $\cZ_F$ containing the folds (equivalently, the edges of the component tree $C(\cP_F)$), and the  slant sums described above occur in the connected components of $\cP \ominus F$ (the vertices of $C(\cP_F)$). Thus the component tree $C(\cP_F)$ is a path, as desired.
\end{proof}

In view of Lemma~\ref{lemma: generalized mobiles have path folding} and Proposition~\ref{prop:goodordering}, we choose the path order $\sigma$ on $F$ that is consistent with the ordered set $S$ of the underlying ribbon $\cZ_S$ of $\cP$. In this case, we call $\sigma$ the {\em induced path order}. Because this $\sigma$ is a path order, the associated component array $M_{\sigma}(\cP_F)$ has connected poset entries; we characterize these poset entries in the next result.
 
\begin{lemma}\label{cor:path_ordering_gives_d_complete_posets}
Let $\cP$ be a mobile poset and $F$ be the set of path folds for $\cP$ with induced path order $\sigma$.  Then every entry of the component array $M_\sigma(\cP_F)$ is a $d$-complete poset.
\end{lemma}

\begin{proof}
The poset $\cP_F$ is a mobile with respect to the folded ribbon $\cZ_F$. As in the proof of Lemma~\ref{lemma: generalized mobiles have path folding}, we have that $\cZ_F$ is a rooted tree; moreover, the maximal element of $\cZ_F$ coincides with the anchor $z'$ if $\cP$ is not free-standing with respect to $\cZ$. Also by Lemma~\ref{lemma: generalized mobiles have path folding}, we have that $C(\cP_F)$ is a path, so every poset $\cP_{i,j} := M_{\sigma}(\cP_F)_{i,j}$ comes from a subpath in $C(\cP_F)$. Thus $\cP_{i,j}$ is a mobile with respect to a smaller ribbon, denote it by $\cZ_{i,j}$, that has a unique maximal element. Moreover, this maximal element coincides with the anchor $z'$ of $\cP$, if $\cP_{i,j}$ is not free-standing with respect to $\cZ_{i,j}$. By Proposition~\ref{prop:when-is-mobile-dcomplete}, it follows that each mobile $\cP_{i,j}$ is a  $d$-complete poset, as desired.
\end{proof}

We now state and prove the main result for computing linear extensions of mobile posets.

\begin{theorem} \label{thm:gmobiledetdone}
Let $\cP$ be a mobile poset with $n$ elements and $F$ be the set of path folds for $\cP$ with induced path order $\sigma$.  Then 
\begin{equation}\label{eqn:gmobiledet}
e(\cP) \,=\, n!\cdot\det(M_{i,j})_{0 \leq i,j\leq k}, \quad \text{for} \quad 
M_{i,j} := \begin{cases}0 & \text{if } j<i-1,\\
1 & \text{if } j = i-1,\\
1/\prod_{x \in \cP_{i,j}} h_{\cP_{i,j}}(x) & \text{otherwise},
\end{cases}
\end{equation}
where $k$ is the size of $F$ and $\cP_{i,j}$ is the connected $d$-complete poset $(M_\sigma(\cP_F))_{i,j}$. 
\end{theorem}

\begin{proof}
Since $\sigma$ is a path order, we may apply Lemma~\ref{thm:det} to get
\[
e(\cP) = n! \cdot \det(\overline{e}(M_{\sigma}(\cP_F))).
\]
Thanks to Lemma~\ref{cor:path_ordering_gives_d_complete_posets}, each poset $\cP_{i,j} := (M_\sigma(\cP_F))_{i,j} = C(\cP_F)[i,j]$ for $i \leq j$ is a connected $d$-complete poset, so $\overline{e}(\cP_{i,j})$ is given by the hook-length formula $1/\prod_{x \in \cP_{i,j}} h_{\cP_{i,j}}(x)$ via Theorem~\ref{thm:dhook}.
\end{proof}

\begin{example} \label{ex:genmobile}
Consider the mobile poset $\cP$ with Hasse diagram and set of path folds $F = \{(e,b), (f,d)\}$ pictured in Figure~\ref{fig: component tree and array gen mobile}: Left. The component tree $C(\cP_F)$ and the component array $M_{\sigma}(\cP_F)$ are pictured in Figure~\ref{fig: component tree and array gen mobile}: Center, Right.

Applying Theorem~\ref{thm:gmobiledetdone} to $\cP$ gives the determinantal formula
\[e(\cP) = 10!\cdot\det(\overline{e}(M_{\sigma}(\cP_F))) = 10!\cdot\det
\begin{pmatrix}
\frac{1}{1} & \frac{1}{9\cdot8\cdot5\cdot3\cdot2\cdot2\cdot2} & \frac{1}{10\cdot9\cdot6\cdot3\cdot2\cdot2\cdot2} \\[6pt]
1 & \frac{1}{8\cdot7\cdot5\cdot3\cdot2\cdot2} & \frac{1}{9\cdot8\cdot6\cdot3\cdot2\cdot2} \\[6pt]
0 & 1 & \frac{1}{1}
\end{pmatrix}
= 240. 
\]

\renewcommand{\qedsymbol}{\rule{0.65em}{0.65em}}\hfill\qedsymbol

\begin{figure}
    \centering
   \raisebox{30pt}{\includegraphics[scale=0.75]{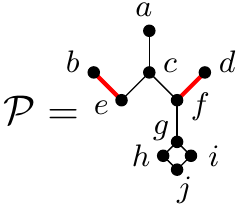}} \,\,
    \raisebox{20pt}{\includegraphics[scale=0.75]{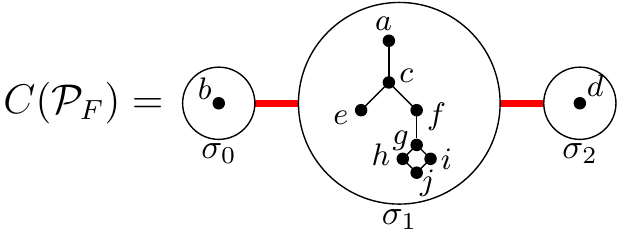}} \,\,    \includegraphics[scale=0.75]{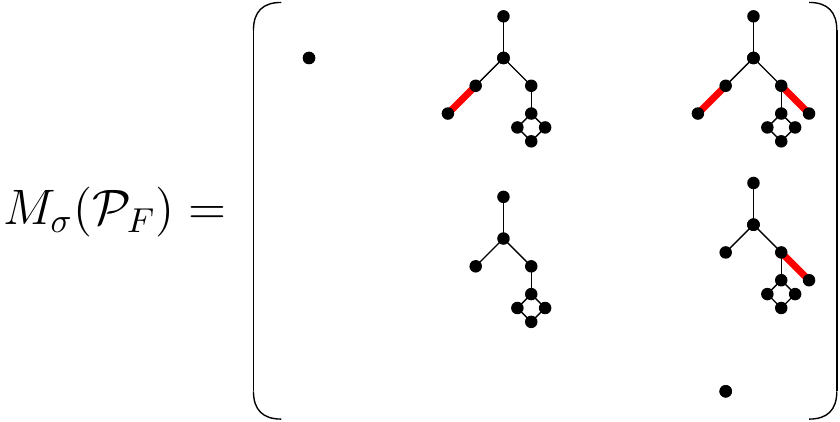}
    \caption{Left: a mobile poset $\cP$ with folds highlighted. Center: its component tree $C(\cP_F)$, with path order $\sigma = (\sigma_0,\sigma_1,\sigma_2)$. Right: its component array $M_\sigma(\cP_F)$.}
    \label{fig: component tree and array gen mobile}
\end{figure}
\end{example}

\subsubsection{Mobile tree posets}

In this subsection and the next, we give examples and applications of Theorem~\ref{thm:gmobiledetdone}.  To illustrate these examples, we restrict to a smaller and simpler class of posets called ``mobile tree posets"---there are three reasons for this: here the entries of the component array are rooted trees, we are able to give a complete classification of this class of posets in Theorem~\ref{cor:mobile characterization}, and mobile tree posets admit both major index and inversion $q$-analogues of our determinantal formulas  while for the more general class of mobile posets, we were only able to obtain a major index $q$-analogue (see Remark~\ref{rem:noinv}). 

\begin{definition}\label{def:mobiles}
A {\em mobile tree poset} (resp. {\em free-standing mobile tree poset}) $\cP$ is a mobile poset (resp. free-standing mobile poset) for which all $\cR_z^{(i)}$ in Operation (i) and $\cQ_{z'}$ in Operation (ii) (resp. all $\cR_z^{(i)}$ in Operation (i)) of Definition~\ref{def:generalized mobiles} are required to be rooted trees.\footnote{Note that every element of a rooted tree is acyclic (see Section~\ref{section:d-complete}), so any element of the rooted tree poset $\cQ_{z'}$ can be chosen to cover the anchor $z'$ in $\cZ$.} 
\end{definition}

See Figure~\ref{fig:mobile tree poset}: Right for a schematic of a mobile tree poset. Note that the class of mobile tree posets still contains ribbons and rooted tree posets.  

As a corollary of Theorem~\ref{thm:gmobiledetdone}, we have a determinant formula for the number of linear extensions of mobile tree posets. 

\begin{corollary} \label{thm:mobiledetdone}
Let $\cP$ be a mobile tree poset with $n$ elements and $F$ be the set of path folds for $\cP$ with induced path order $\sigma$.  Then 
\begin{equation}
e(\cP) \,=\, n!\cdot\det(M_{i,j})_{0 \leq i,j\leq k}, \quad \text{for} \quad 
M_{i,j} := \begin{cases}0 & \text{if } j<i-1,\\
1 & \text{if } j = i-1,\\
1/\prod_{x \in \cP_{i,j}} h_{\cP_{i,j}}(x) & \text{otherwise},
\end{cases}
\end{equation}
where $k$ is the size of $F$
and $\cP_{i,j}$ is the rooted tree poset $(M_\sigma(\cP_F))_{i,j}$. 
\end{corollary}

\begin{proof}
A mobile tree poset is an example of a mobile poset, so we apply Theorem~\ref{thm:gmobiledetdone} and note that the posets in the component array $M_\sigma(\cP_F)$, denoted by $\cP_{i,j}:= (M_{\sigma}(\cP_F))_{i,j}$,
are rooted trees.
\end{proof}

\begin{example} \label{ex:mobile}
Consider the mobile tree poset $\cP$ with Hasse diagram and set of path folds $F=\{(d,a),(d,c)\}$ pictured in Figure~\ref{fig: component tree and array2}: Left. The component tree $C(\cP_F)$ and the component array $M_{\sigma}(\cP_F)$ are illustrated in Figure~\ref{fig: component tree and array2}: Center, Right.

Applying Corollary~\ref{thm:mobiledetdone} to $\cP$ yields the determinantal formula
\[
e(\cP) = 6! \cdot \det(\overline{e}(M_{\sigma}(\cP_F))) = 6! \cdot \det
\begin{pmatrix}
\frac{1}{1} & \frac{1}{5\cdot 4} & \frac{1}{6\cdot 5} \\[6pt]
1 & \frac{1}{4\cdot 3} & \frac{1}{5\cdot 4} \\[6pt]
0 & 1 & \frac{1}{1}
\end{pmatrix}
= 12.
\]
\renewcommand{\qedsymbol}{\rule{0.65em}{0.65em}}\hfill\qedsymbol
\begin{figure}
    \centering
   \raisebox{35pt}{\includegraphics[scale=0.8]{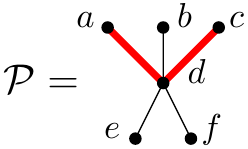}} \quad
    \raisebox{22pt}{\includegraphics[scale=0.8]{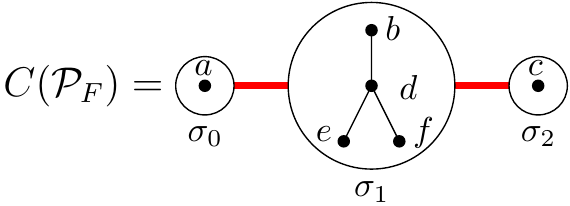}} \,\,\,    \includegraphics[scale=0.8]{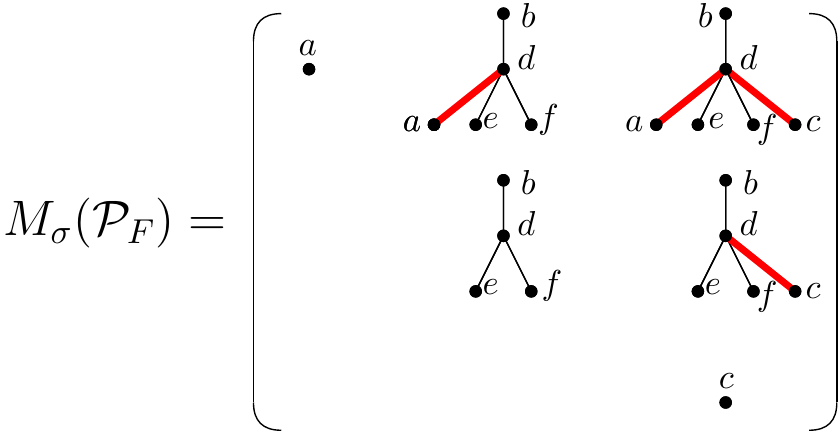}
    \caption{Left: a mobile tree poset $\cP$ with folds highlighted. Center: its component tree $C(\cP_F)$, with path order $\sigma = (\sigma_0,\sigma_1,\sigma_2)$. 
    Right: its component array $M_\sigma(\cP_F)$.}
    \label{fig: component tree and array2}
\end{figure}
\end{example}

In fact, the result below implies that mobile tree posets are the only tree posets for which there exists a set of folds where the component tree is a path and the folded poset is a rooted tree; thus, we have the following characterization of mobile tree posets.

\begin{theorem}\label{cor:mobile characterization}
A tree poset $\cP$ is a mobile tree poset if and only if there exists a set of folds $F$ such that $\cP_F$ is a rooted tree and $C(\cP_F)$ is a path.
\end{theorem}

\begin{proof}
Suppose $\cP$ is a mobile tree poset and let $F$ be the path folds from \eqref{eq: folds mobile tree poset}.  By Lemma~\ref{lemma: generalized mobiles have path folding}, we have that $C(\cP_F)$ is a path.  Showing that $\cP_F$ is a rooted tree is a straightforward consequence of the proof of Lemma~\ref{cor:path_ordering_gives_d_complete_posets}---the only change is that the $\cR_z^{(i)}$ in Step (a) and $\cQ_{z'}$ in Step (b) are rooted tree posets, and it is clear that the outcome of doing all of the corresponding slant sums is a rooted tree poset.

Conversely, let $\cP$ be a tree poset with a set of folds $F=\{f_1,\ldots,f_k\}$ such that $\cP_F$ is a rooted tree and $C(\cP_F)$ is a path. Then each of the vertices of $C(\cP_F)$ corresponds to a rooted subtree in both $\cP_F$ and $\cP$. We denote these rooted trees by $\cT_0,\ldots, \cT_k$ (following one of the two orders of the path $C(\cP_F)$). Let $\cZ$ be the path in $\cP$ obtained by completing the subpath $F$ with the unique subpath in the tree $\cT_i$ from $f_i$ to $f_{i+1}$ for $i=1,\ldots,k-1$. This path $\cZ$ as a subposet of $\cP$ is a ribbon poset. 

Since $\cP_F$ is a rooted tree, given an element $z$ in $\cZ$, the elements in $\cP \setminus \cZ$ 
smaller than it are a disjoint sum of rooted trees $\{ \cT^{(i)}_{z}\}$. The element $z$ covers the root of each of these rooted trees. This coincides with the outcome of performing Operation (i) in the definition of mobile tree posets (see Definitions~\ref{def:mobiles} and~\ref{def:generalized mobiles}).

In addition, there is at most one cover relation $z'\lessdot x$, where $z'$ is in $\cZ$ and $x$ is in $\cP\setminus \cZ$. Since if there were another such cover relation $z''\lessdot y$, neither of these relations is folded in $\cP_F$ (because folds only occur in $\cZ$), so either $\cP_F$ has more than one maximal element or the Hasse diagram of $\cP_F$ has a cycle.  Both of these situations contradict the fact that $\cP_F$ is a rooted tree. 
We also know that the elements in the same connected component as $x$ in the Hasse diagram of $\cP \setminus \cZ$ form a rooted subtree (because $\cP_F$ is a rooted tree) that we denote by $\cQ_{z'}$. The element $z'$ is covered by $x$, which can be any element of $\cQ_{z'}$. This coincides with Operation (ii) in the definition of mobile tree posets (see Definitions~\ref{def:mobiles} and~\ref{def:generalized mobiles}), where we view the element $z'$ as the anchor. It follows that $\cP$ is a mobile tree poset, as desired. 
\end{proof}

\begin{remark}
There are posets that cannot be expressed as mobile posets, but for which  Lemma~\ref{thm:det} still applies (see  Figure~\ref{fig: example and nonexample of mobile}(d), where the folds are highlighted). 
 \end{remark}

\subsubsection{Two new generalizations of Euler numbers}
\label{ex:eulerext}

The even Euler ribbon numbers $E_{2k}$ count the linear extensions of the ribbon poset $\cZ=\cZ^{(2k)}_{\{2,4,\ldots,2k-2\}}$, which is the infinite family of up-down posets. We illustrate Corollary~\ref{thm:mobiledetdone} by generalizing the even Euler ribbon numbers using two different families of posets:

\begin{equation} \label{eq: two families of mobiles}
    \includegraphics[scale=0.95]{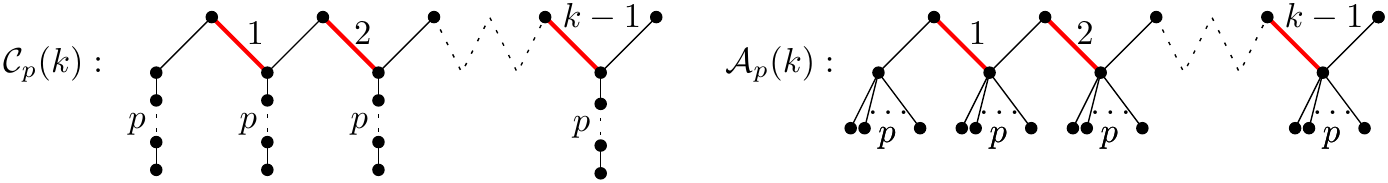}
\end{equation}

The first variation is obtained by appending a chain of $p$ elements to each minimal element of the original up-down poset $\cZ$. We denote the resulting poset by $\mathcal{C}_p(k)$. For example, the posets $\mathcal{C}_1(k)$ for $k=1,2,3,4$ are the following:  \begin{center}
\includegraphics[scale=0.8]{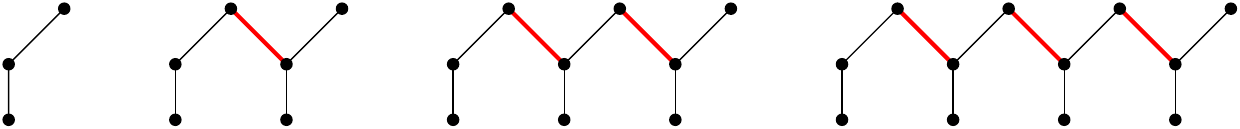}
\end{center}

\begin{corollary} \label{cor:first-variation}
The number of linear extensions of the poset $\mathcal{C}_p(k)$ is
\[
e(\mathcal{C}_p(k)) \,=\, ((p+2)k)! \cdot \det(c_{i,j})_{1\leq i,j\leq k}, 
\]
where
\[
c_{i,j} = \begin{cases}
\prod_{r=1}^{j-i+1} 1/(p!(rp+2r-1)(rp+2r)) &\text{ if } j\geq i-1, \\
0 & \text{ otherwise.}
\end{cases}
\]
\end{corollary}

\begin{proof}
This follows by applying Corollary~\ref{thm:mobiledetdone} to this family and computing the hook lengths of the entries of the component array.
\end{proof}

We recover the determinant formula for the even Euler ribbon numbers $E_{2k} = e(\cZ^{(2k)}_{\{2,4,\ldots,2k-2\}})$ given by \eqref{eq:Detribbon} when $p=0$ above, and when $p=1$ we obtain the sequence \cite[\href{https://oeis.org/A332471}{A332471}]{OEIS} 
\begin{equation}\label{eq: sequence first family}
   \{e(\mathcal{C}_1(k))\}_{k=1}^{\infty} = \{1, 16, 1036, 174664, 60849880,\dots\}
\end{equation}
with determinant formula:
\[
e(\mathcal{C}_1(k)) = (3k)!\cdot \det(a_{i,j})_{1\leq i,j\leq k}, \quad \text{ where } a_{i,j} = \begin{cases}
\prod_{r=1}^{j-i+1}\frac{1}{3r(3r-1)} &\text{ if } j\geq i-1, \\
0 & \text{ otherwise.}
\end{cases}
\]

The second variation is obtained by appending an antichain of $p$ elements to each minimal element of the original up-down poset $\cZ$. We denote the resulting poset by $\mathcal{A}_p(k)$. For example, the posets $\mathcal{C}_1(k)$ and $\mathcal{A}_1(k)$ are isomorphic, and the posets $\mathcal{A}_2(k)$ for $k=1,2,3,4$ are the following:
\begin{center}
\includegraphics[scale=0.8]{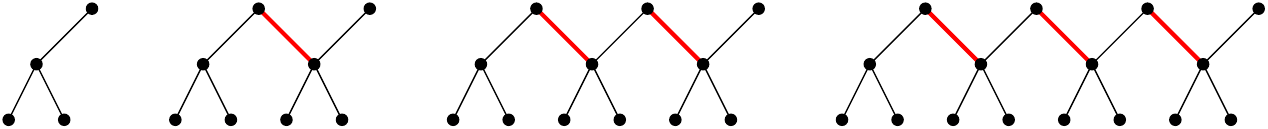}
\end{center}

\begin{corollary} \label{cor:second-variation}
The number of linear extensions of the poset $\mathcal{A}_p(k)$ is
\[
e(\mathcal{A}_p(k)) \,=\, ((p+2)k)! \cdot \det(a_{i,j})_{1\leq i,j\leq k}, 
\]
where
\[
a_{i,j} = \begin{cases}
\prod_{r=1}^{j-i+1} 1/((rp+rk-1)(rp+2r)) &\text{ if } j\geq i-1, \\
0 & \text{ otherwise.}
\end{cases}
\]
\end{corollary}

\begin{proof}
This follows by applying Corollary~\ref{thm:mobiledetdone} to this family and computing the hook lengths of the entries of the component array.
\end{proof}

When $p=2$, we obtain the sequence \cite[\href{https://oeis.org/A332568}{A332568}]{OEIS} 
\begin{equation}\label{eq: sequence second family}
\{e(\mathcal{A}_2(k))\}_{k=1}^{\infty} = \{2, 220, 163800, 445021200, 3214652032800, \ldots\}
\end{equation}
with determinant formula
\[
e(\mathcal{A}_2(k)) = (4k)!\cdot \det(c_{i,j})_{1\leq i,j\leq k}, \quad \text{ where } c_{i,j} = \begin{cases}
\prod_{r=1}^{j-i+1}\frac{1}{4r(4r-1)} &\text{ if } j\geq i-1, \\
0 & \text{ otherwise.}
\end{cases}
\]

\section{Determinant formulas for \texorpdfstring{$q$}{q}-analogues of linear extensions of mobiles}
\label{section:q}

In this section, we obtain determinant formulas for the major index $q$-analogue of $e(\cP)$ for mobile posets (and therefore also for mobile tree posets as a special case) and for the inversion $q$-analogue of $e(\cP)$ for mobile tree posets. The reason for not proving an inversion $q$-analogue for mobile posets is that we do not know of an inversion $q$-analogue of Theorem~\ref{lemma:qhookmaj d complete} for $d$-complete posets (see Remark~\ref{rem:noinv}).  

Throughout this section, we use the preliminaries introduced in Section~\ref{section:qanalogues}.

\subsection{Inclusion-exclusion for \texorpdfstring{$q$}{q}-analogues of linear extensions}

In this section, we collect results which hold for any $\stat$ in $\{\maj,\inv\}$.  A subposet $\cQ$ of a labeled poset $(\cP,\omega)$ gets its own labeling by restricting the labeling $\omega$ on $\cP$ to $\cQ$; in this situation, we will often abuse notation and write $(\cQ,\omega)$ for the resulting labeled poset when the meaning is clear from context.
 
 \begin{lemma}\label{lem:qincexc}
 Let $(\cP,\omega)$ be a labeled poset, $(x, y)$ be in $\lessdot_{\cP}$, and $\stat$ be in $\{\maj,\inv\}$.  Then
 \[
 e^\stat_q(\cP, \omega) = e^\stat_q(\cP \ominus \{(x, y)\} , \omega) - e^\stat_q(\cP_{\{(x,y)\}}, \omega).
 \]
 \end{lemma}

 \begin{proof}
 The equality of sets of linear extensions in \eqref{linext IE decomposition} extends to linear extensions of labeled posets:
 \begin{equation} \label{linext IE decomposition labeled posets}
\cL(\cP,\omega) = \cL(\cP \ominus \{(x,y)\},\omega) \,\setminus\, \cL(\cP_{\{(x,y)\}},\omega).
\end{equation}
By taking the generating polynomial over the sets on both sides with respect to the statistic $\stat \in \{\maj, \inv\}$, the result follows.
\end{proof}
 
 \begin{corollary}\label{cor:alt}
Let $(\cP,\omega)$ be a labeled poset, $F \subset\ \lessdot_\cP$, and $\stat$ be in $\{\maj,\inv\}$.  Then
\begin{equation}
e_q^\stat(\cP,\omega) = \sum_{S\subset F}(-1)^{\#S} e_q^\stat(\cP_{S,F},\omega).
\end{equation}
\end{corollary}

\begin{proof}
The result follows from repeated application of Lemma~\ref{lem:qincexc}.
\end{proof}

Given a labeled poset $(\cP,\omega)$ and a set of folds $F$, component arrays are defined for labeled posets analogously and are denoted by $M_\sigma(\cP_F,\omega)$, where $\sigma$ is any total order on the vertices of $C(\cP_F)$.  Given a component array $M_\sigma(\cP_F,\omega)$ with $\sigma$ a path order on the vertices of $C(\cP_F)$, we define a matrix $\overline{e}_q^\stat(M_{\sigma}(\cP_F,\omega))$ by
\begin{equation}\label{qcomparray}
\overline{e}_q^\stat(M_{\sigma}(\cP_F,\omega))_{i,j} := \left\{\begin{array}{ll}0 & \text{if } j<i-1 ,\\
1 & \text{if } j = i-1,\\
\overline{e}_q^\stat((M_\sigma(\cP_F,\omega))_{i,j}) & \text{otherwise}, \end{array}\right.
\end{equation}
where $(i,j)\in [0,k]\times [0,k]$.

\subsection{\texorpdfstring{$q$}{q}-analogue by major index for mobile posets}
The main result of this section is Theorem~\ref{thm:qmajgmobiledetdone}, which gives a major index $q$-analogue of Theorem~\ref{thm:gmobiledetdone}.  Note that this also gives a major index $q$-analogue of Corollary~\ref{thm:mobiledetdone}, since mobile tree posets are a special case of mobile posets. 

\begin{lemma}\label{thm:qdet-maj}
Let $(\cP,\omega)$ be a labeled poset with $n$ elements, and let $\sigma$ be a path order on the vertices of $C(\cP_F)$. Then we have
\[
e_q^\maj(\cP,\omega) = [n]_q!\cdot\det(\overline{e}_q^{\,\maj}(M_{\sigma}(\cP_F,\omega))).
\]
\end{lemma}

\begin{proof}
The proof is analogous to the proof of Lemma~\ref{thm:det}.  That is, we apply Proposition~\ref{prop:qdjsum-maj} to each term of the alternating sum in Corollary~\ref{cor:alt}.  Then we use Lemma~\ref{lemma:det} with $g(i,j+1) = \overline{e}_q^{\,\maj}((M_\sigma(\cP_F,\omega))_{i,j})$ for $0\leq i \leq j \leq k$, $g(i,i)=1$, and $g(i,j)=0$ for $j<i$.
\end{proof}
 
We now present our $q$-analogue of Theorem \ref{thm:gmobiledetdone}.

\begin{theorem}\label{thm:qmajgmobiledetdone}
Let $(\cP,\omega)$ be a labeled mobile poset with $n$ elements and $F$ be the set of path folds for $\cP$ with induced path order $\sigma$.  Then
\begin{equation}
e_q^\maj(\cP, \omega) \,=\, [n]_q!\cdot\det(M_{i,j})_{0 \leq i,j\leq k}, \quad \text{for} \quad 
M_{i,j} := \begin{cases}0 & \text{if } j<i-1,\\
1 & \text{if } j = i-1,\\
\frac{q^{\maj(\cP_{i,j},\omega_{i,j})}}{\prod_{x \in \cP_{i,j}} [h_{\cP_{i,j}}(x)]_q} & \text{otherwise},
\end{cases}
\end{equation}
where $k$ is the size of $F$ and $(\cP_{i,j},\omega_{i,j})$ is the labeled connected $d$-complete poset $(M_\sigma(\cP_F,\,\omega))_{i,j}$.
\end{theorem}
\begin{proof}
Since $\sigma$ is a path order, we may apply Lemma~\ref{thm:qdet-maj} to get
\[
e_q^{\maj}(\cP,\omega) = [n]_q! \cdot \det(\overline{e}_q^{\,\maj}(M_{\sigma}(\cP_F,\,\omega))).
\]
Each poset $\cP_{i,j} := (M_\sigma(\cP_F))_{i,j} = C(\cP_F)[i,j]$ for $i \leq j$ is a connected $d$-complete poset (see Lemma~\ref{cor:path_ordering_gives_d_complete_posets}), so $\overline{e}_q^{\maj}(\cP_{i,j},\omega_{i,j}) = q^{\maj(\cP_{i,j}, \omega_{i,j})}/\prod_{x \in \cP_{i,j}} [h_{\cP_{i,j}}(x)]_q$ via Theorem~\ref{lemma:qhookmaj d complete}.
\end{proof}

We present two examples of using Theorem~\ref{thm:qmajgmobiledetdone} to compute $e_q^\maj(\cP,\omega)$ for a mobile tree poset and a mobile poset, respectively.

\begin{example}\label{ex:qmaj}
Let $\cP$ and $M_{\sigma}(\cP_F)$ be the same poset and component array from Example \ref{ex:mobile} (see Figure~\ref{fig: component tree and array2}) with labeling $\omega$ on $\cP$ given by
\[
a\mapsto 1,\quad b \mapsto 3,\quad c \mapsto 6,\quad d\mapsto 4,\quad e \mapsto 2,\quad f \mapsto 5.
\]
Applying Theorem~\ref{thm:qmajgmobiledetdone} to this labeled mobile tree poset $(\cP,\omega)$ yields the determinantal formula
\[
 e^{\maj}_q(\cP,\omega) = [6]_q! \cdot \det 
    \begin{pmatrix}
        \frac{1}{[1]_q} & \frac{q^5}{[5]_q[4]_q} & \frac{q^7}{[6]_q[5]_q} \\[6pt]
        1 & \frac{q^4}{[4]_q[3]_q} & \frac{q^6}{[5]_q[4]_q} \\[6pt]
        0 & 1 & \frac{1}{[1]_q}
    \end{pmatrix}
     = q^{11}+3q^{10}+3q^9+q^8+q^6+2q^5+q^4.
\]
\renewcommand{\qedsymbol}{\rule{0.65em}{0.65em}}\hfill\qedsymbol
 \end{example}

\begin{example}\label{ex:qmaj2}
Let $\cP$ and $M_{\sigma}(\cP_F)$ be the same poset and component array from Example \ref{ex:genmobile} (see Figure~\ref{fig: component tree and array gen mobile}) with labeling $\omega$ on $\cP$ given by 
\[
a \mapsto 10,\quad b \mapsto 7,\quad c \mapsto 8,\quad d \mapsto 9,\quad e \mapsto 5,\quad f \mapsto 6,\quad g \mapsto 4,\quad h \mapsto 2,\quad i \mapsto 3,\quad j \mapsto 1.
\]
Applying Theorem~\ref{thm:qmajgmobiledetdone} to this labeled mobile poset $(\cP,\omega)$ yields the determinantal formula
\[
 e^{\maj}_q(\cP,\omega) = [10]_q! \cdot \det 
    \begin{pmatrix}
     \frac{1}{[1]} & \frac{q}{[9][8][5][3][2][2][2]} & \frac{q^2}{[10][9][6][3][2][2][2]} \\[6pt]
        1 & \frac{1}{[8][7][5][3][2][2]} & \frac{q}{[9][8][6][3][2][2]} \\[6pt]
        0 & 1 & \frac{1}{[1]}
    \end{pmatrix},
\]
where we have suppressed $q$ from the notation $[m]_q$ in the matrix entries for space concerns.  The powers of $q$ in the numerators of the entries of the matrix are due to the two possible descents in $\Des(\cP_F, \omega)$: $b \lessdot e$ and $d \lessdot f$.\renewcommand{\qedsymbol}{\rule{0.65em}{0.65em}}\hfill\qedsymbol
 \end{example}

\subsection{\texorpdfstring{$q$}{q}-analogue for inversions of mobile tree posets}

The main result of this section is Theorem~\ref{thm:qmobiledetdone}, which gives an inversion $q$-analogue of Corollary~\ref{thm:mobiledetdone}. In contrast with the major index $q$-analogue in Theorem~\ref{thm:qmajgmobiledetdone}, the inversion $q$-analogue is more delicate and requires a special labeling of the poset.

\begin{definition}\label{def:prl}
Let $\cP$ be a mobile tree poset, and let $F$ be the set of path folds for $\cP$ with induced path order $\sigma$.  Then $\sigma$ gives an order $\cP_{\sigma_0}, \cP_{\sigma_1}, \ldots, \cP_{\sigma_k}$ on the connected components of the poset $\cP \ominus F$.  A labeling $\omega$ on $\cP$ is called a \emph{$\sigma$-partitioned labeling} if whenever $\sigma_i < \sigma_j$, we have
\[
\omega(x) < \omega(y) \qquad \text{for every} \qquad  x \in \cP_{\sigma_i}, \,\,y \in \cP_{\sigma_j}.
\]
Moreover, $\omega$ is called a \emph{$\sigma$-partitioned regular labeling} if it is a $\sigma$-partitioned labeling such that the restriction of $\omega$ to each connected component $\cP_{\sigma_i}$ of $\cP \ominus F$ is a regular labeling of that component.
\end{definition}

\begin{proposition}\label{prop:partreglab}
Let $\cP$ be a mobile tree poset and $F$ be the set of path folds for $\cP$ with induced path order $\sigma$.  Then there exists a $\sigma$-partitioned regular labeling of $\cP$.
\end{proposition}

\begin{proof} Lemma~\ref{cor:path_ordering_gives_d_complete_posets}, together with the first paragraph of the proof of Theorem~\ref{cor:mobile characterization}, asserts that the posets $C(\cP_F)[i,i]$ for $i=0,\ldots,k$ (that is, the connected components of $\cP \ominus F$) are rooted trees.  Assume that the $i$th rooted tree has $n_i$ elements for $i=0,\ldots,k$.  Since rooted trees are two-dimensional posets, they admit a regular labeling \cite[Theorem 6.9]{reglab}, so we give each of the rooted tree posets $C(\cP_F)[0,0],\ldots, C(\cP_F)[k,k]$ a regular labeling with the labels
 \[
 [n_0],\,\, [n_0+n_1]\,\setminus\, [n_0],\,\,\ldots,\,\,[n_0+\cdots + n_k] \,\setminus\, [n_0+\cdots + n_{k-1}], 
 \]
 respectively. This yields a $\sigma$-partitioned regular labeling of $\cP$.
\end{proof}

\begin{example}
Let $\cP$ be the mobile tree poset illustrated below, with set of path folds $F$ colored red and labeling $\omega$ indicated.  To the right of the poset, we depict the induced path order $\sigma$ on the vertices of $C(\cP_F)$. 

\begin{center}
    \raisebox{8pt}{\includegraphics[scale=0.8]{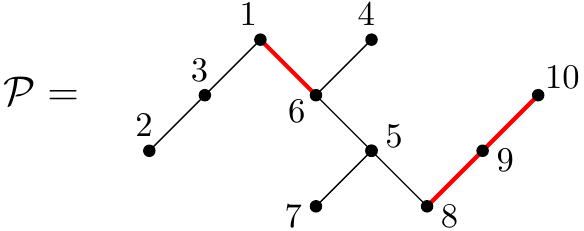}}
    \qquad 
    \includegraphics[scale=0.8]{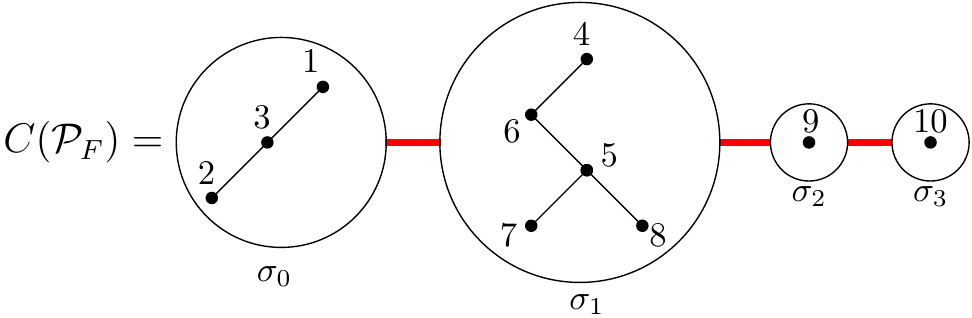}
\end{center}
Note that $\omega$ is a $\sigma$-partitioned regular labeling of $\cP$.\renewcommand{\qedsymbol}{\rule{0.65em}{0.65em}}\hfill\qedsymbol
\end{example}

\begin{lemma}\label{thm:qdet-inv}
Let $(\cP,\omega)$ be a labeled mobile tree poset with $n$ elements, $F$ be the set of path folds for $\cP$ with induced path order $\sigma$, and $\omega$ be a $\sigma$-partitioned labeling of $\cP$. 
Then we have
\[
e_q^\inv(\cP,\omega) = [n]_q!\cdot\det(\overline{e}_q^{\,\inv}(M_{\sigma}(\cP_F,\omega))).
\]
\end{lemma}

\begin{proof}
The proof is analogous to the proof of Lemma~\ref{thm:det}---the only new concern is that the posets $(\cP_{S,F},\omega)$ should satisfy the hypotheses of Proposition~\ref{prop:qdjsum-inv} with respect to the decomposition in \eqref{eqn:djsum}, but this follows from the fact that the posets appearing in the right-hand side of \eqref{eqn:djsum} are defined by intervals which are coarsenings of the intervals $[0,0], [1,1], \ldots, [k,k]$ that define the connected components of $\cP \ominus F$.  Thus, we may apply Proposition~\ref{prop:qdjsum-inv} to each term of the alternating sum in Corollary~\ref{cor:alt} and conclude the proof by using Lemma~\ref{lemma:det} with $g(i,j+1) = \overline{e}_q^{\,\inv}((M_\sigma(\cP_F,\omega))_{i,j})$ for $0\leq i \leq j \leq k$, $g(i,i)=1$, and $g(i,j)=0$ for $j<i$.
\end{proof}
 
\begin{lemma}\label{lemma:prl}
Let $(\cP,\omega)$ be a labeled mobile tree poset with $n$ elements, $F$ be the set of path folds for $\cP$ with induced path order $\sigma$, and $\omega$ be a $\sigma$-partitioned regular labeling of $\cP$.  Then the restriction of the labeling $\omega$ to each entry of the component array $M_\sigma(\cP_F)$ is a regular labeling.
\end{lemma}

\begin{proof} 
We use the same reduction as in the proof of Lemma~\ref{cor:path_ordering_gives_d_complete_posets}; that is, every poset $\cP_{i,j}:=(M_{\sigma}(\cP_F))_{i,j}$ comes from a subpath in $C(\cP_F)$ and is a smaller mobile tree poset that is actually a rooted tree poset. Also, the restriction of $\omega$ to each $\cP_{i,j}$ is a $(\sigma_i,\sigma_{i+1},\ldots,\sigma_j)$-partitioned regular labeling of the unfolded poset corresponding to $\cP_{i,j}$.  Thus, it suffices to show that $\omega$ is a regular labeling for the poset $\cP_F$.

Since $\omega$ is a $\sigma$-partitioned regular labeling of $\cP$, it follows that the restriction of $\omega$ to each component $C(\cP_F)[i,i]$ is a regular labeling.  Thus, it remains to show that $\omega$ is regular across connected components.  Since every element in $C(\cP_F)[i,i]$ has a smaller label than every element in $C(\cP_F)[i+1,i+1]$, the orientation of the edge in $F$ between the two connected components does not affect the regularity of the labeling. Thus, $\omega$ is a regular labeling of $\cP_F$ as desired.
\end{proof}

\begin{theorem}\label{thm:qmobiledetdone}
Let $(\cP,\omega)$ be a labeled mobile tree poset with $n$ elements, $F$ be the set of path folds for $\cP$ with induced path order $\sigma$, and $\omega$ be a $\sigma$-partitioned regular labeling of $\cP$. Then
\begin{equation}
e_q^{\inv}(\cP, \omega) \,=\, [n]_q!\cdot\det(M_{i,j})_{0 \leq i,j\leq k}, \quad \text{for} \quad 
M_{i,j} := \begin{cases}0 & \text{if } j<i-1,\\
1 & \text{if } j = i-1,\\
\frac{q^{\inv(\cP_{i,j}, \omega_{i,j})}}{\prod_{x \in \cP_{i,j}} [h_{\cP_{i,j}}(x)]_q} & \text{otherwise},
\end{cases}
\end{equation}
where $k$ is the size of $F$ and $(\cP_{i,j},\omega_{i,j})$ is the labeled rooted tree poset $(M_\sigma(\cP_F,\,\omega))_{i,j}$. 
\end{theorem}
\begin{proof}
Lemma~\ref{thm:qdet-inv} assures that
\[
e_q^{\inv}(\cP,\,\omega) \,=\, [n]_q! \cdot \det(\overline{e}_q^{\,\inv}(M_{\sigma}(\cP_F,\,\omega))).
\]
Each poset $\cP_{i,j} := (M_\sigma(\cP_F))_{i,j} = C(\cP_F)[i,j]$ for $i \leq j$ is a rooted tree with a regular labeling thanks to Lemma~\ref{lemma:prl}, so $\overline{e}_q^{\inv}(\cP_{i,j},\omega_{i,j})$ is given by the hook-length formula $q^{\inv(\cP_{i,j}, \omega_{i,j})}/\prod_{x \in \cP_{i,j}} [h_{\cP_{i,j}}(x)]_q$ via the first statement in Theorem~\ref{lemma:qhook}.
\end{proof}

\begin{example}\label{ex:qinv}
Let $(\cP,\omega)$ be the same labeled mobile tree poset from Example~\ref{ex:qmaj} (see Figure~\ref{fig: component tree and array2}), and note that $\omega$ is a $\sigma$-partitioned regular labeling (where $\sigma$ is the same order depicted in Figure~\ref{fig: component tree and array2}: Center).  Applying Theorem~\ref{thm:qmobiledetdone} to $(\cP,\omega)$ yields the determinantal formula 
\[
 e^{\inv}_q(\cP,\omega) = [6]_q! \cdot \det 
    \begin{pmatrix}
        \frac{1}{[1]_q} & \frac{q^3}{[5]_q[4]_q} & \frac{q^5}{[6]_q[5]_q} \\[6pt]
        1 & \frac{q^3}{[4]_q[3]_q} & \frac{q^5}{[5]_q[4]_q} \\[6pt]
        0 & 1 & \frac{1}{[1]_q}
    \end{pmatrix}
     = q^{10} + 3q^9 + 4q^8 + 3q^7 + q^6.
\]\renewcommand{\qedsymbol}{\rule{0.65em}{0.65em}}\hfill\qedsymbol
 \end{example}

\section{Final Remarks}\label{sec:final_rem}

\subsection{Positive formulas for counting linear extensions of mobiles}\label{subsec:pos_formula}

In Section~\ref{section:mobile}, we generalized a determinant formula for linear extensions of ribbon posets $\cZ$ to mobile posets. It is natural to ask if other known formulas for $e(\cZ)$ generalize to mobile posets. For instance, there is a recent formula by Naruse (see \cites{NO,MPP1}) that computes  $e(\cZ)$, and more generally the number of skew standard Young tableaux, as a positive sum of products of hook lengths. In \cite{NO}, this formula was generalized to skew $d$-complete posets. To define these posets, we recall the notion of an {\em order filter} as a subset $I$ of $\cP$ such that if $x \in I$ and $y\geq_{\cP} x$ then $y\in I$.  

\begin{definition}
A {\em skew $d$-complete poset} is a $d$-complete poset $\cP$ with an order filter $I$ removed.  By abuse of notation, we denote such a poset by $\cP / I$.\footnote{In \cite{NO}, skew $d$-complete posets are denoted by $\cP/I$ to mimic skew shapes $\lambda/\mu$. However, with our conventions from Section~\ref{section:posets} the skew $d$-complete poset $\cP / I$ should be written $\cP\backslash I$, but we stick to the notation established in the literature in this section.}
\end{definition}

The Naruse--Okada formula for counting linear extensions of skew $d$-complete posets is stated in terms of excited diagrams: certain subsets $D$ of the set of elements of $\cP$ that can be obtained from the elements of $I$ by a sequence of certain  excited moves that swap a certain element of $D$ by an element of $\cP \setminus D$. See \cite[Section 3]{NO} and \cite{MPP1} for the precise definition of excited diagrams of skew $d$-complete posets and skew Young diagrams, respectively.

\begin{theorem}[Naruse--Okada \cite{NO}] \label{thm: Naruse Okada HLF}
Let $\cP/I$ be a skew $d$-complete poset with $n$ elements. Then 
\[
e(\cP/I) = n! \sum_{D \in \mathcal{E}(\cP/I)} \prod_{x \in \cP\setminus D} \frac{1}{h_{\cP}(x)},
\]
where $\mathcal{E}(\cP/I)$ is the set of excited diagrams of $\cP/I$.
\end{theorem}

\begin{example}
The poset $\cP_1/I_1 = \mathcal{C}_1(2)$ from Figure~\ref{fig:mobiles and skew d complete}: Left has two excited diagrams $D_1=I_1=\{a\}$ and $D_2=\{e\}$ (see Figure~\ref{fig:excited_diagramsA}).

Their complements have the hook lengths $\{1,1,2,2,3,5\}$ and $\{1,1,2,3,5,6\}$, respectively. In this case, Theorem~\ref{thm: Naruse Okada HLF} gives
\[
e(\mathcal{C}_1(2)) = 6! \left( \frac{1}{2\cdot 2 \cdot 3 \cdot 5} + \frac{1}{2  \cdot 3 \cdot 5 \cdot 6}\right) = 16,
\]
agreeing with the data in \eqref{eq: sequence first family}.\renewcommand{\qedsymbol}{\rule{0.65em}{0.65em}}\hfill\qedsymbol
\end{example}

\begin{figure}
\subfigure{
\includegraphics[scale=0.8]{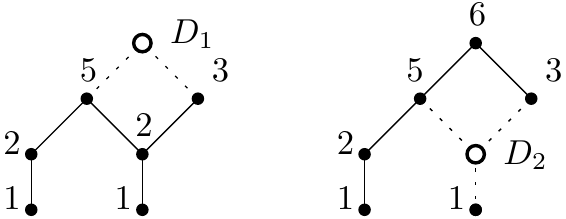} 
\label{fig:excited_diagramsA}
}
\qquad \qquad \qquad
\subfigure{
\includegraphics[scale=0.8]{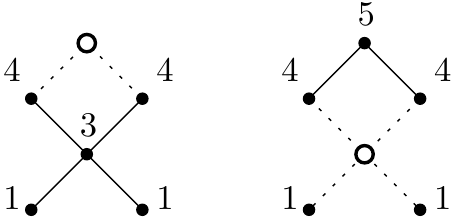} 
\label{fig:excited_diagramsB}
}
\caption{(a) the excited diagrams of $\cP_1/I_1$, (b) possible excited diagrams of $\cP_2/I_2$. The hook lengths are indicated next to the poset elements.}
\end{figure}

The classes of skew $d$-complete posets and mobile (tree) posets are not the same but do have an overlap (see Figure~\ref{fig:mobiles and skew d complete}). It would be interesting to see if skew $d$-complete posets $\cP/I$ also have determinantal identities for the number $e(\cP/I)$ and if mobile (tree) posets $\cQ$ have a Naruse-type formula for $e(\cQ)$. See Example~\ref{ex:NaruseXposet} for small positive evidence of the latter.

\begin{figure}
    \centering
    \includegraphics[scale=0.8]{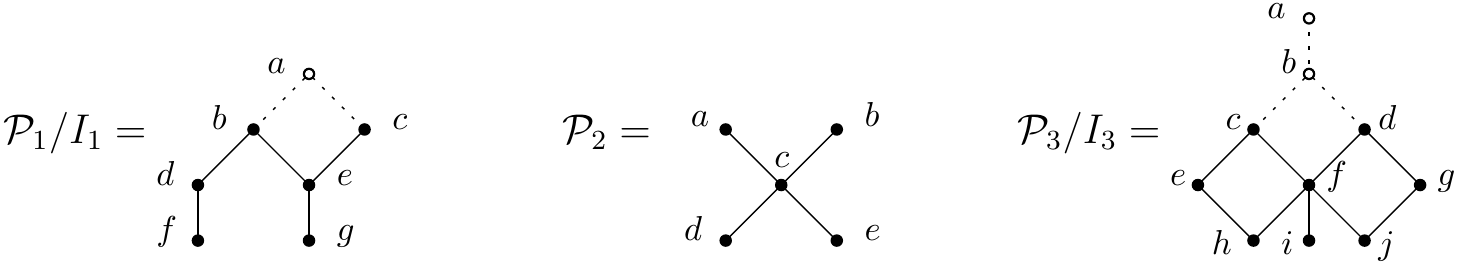}
    \caption{Left: a poset $\cP_1/\{a\}$ that is both skew $d$-complete and a mobile tree poset (isomorphic to $\mathcal{C}_1(2)$ from \eqref{eq: two families of mobiles}). Center: a mobile tree poset $\cP_2$ that is not skew $d$-complete. Right: a skew $d$-complete poset $\cP_3/\{a,b\}$ that is not a mobile.}
    \label{fig:mobiles and skew d complete}
\end{figure}

\begin{example}\label{ex:NaruseXposet}
One can show that the ``$x$ poset'' $\cP_2$ in Figure~\ref{fig:mobiles and skew d complete}, which is a mobile tree poset, is not skew $d$-complete. However, one can write the number $e(\cP_2)=4$ as the Naruse--Okada type sum $5! \left( \frac{1}{3 \cdot 4\cdot 4} + \frac{1}{ 4 \cdot 4 \cdot 5}\right)$ in terms of the two diagrams in Figure~\ref{fig:excited_diagramsB}.
\renewcommand{\qedsymbol}{\rule{0.65em}{0.65em}}\hfill\qedsymbol
\end{example}

\subsection{Extending polynomiality of counting permutations by descents}\label{subsec:desc_poly}

In \cite{CA}, the number of permutations of $[n]$ with a fixed descent set is shown to be a polynomial in $n$. In \cite{diaz2019descent}, the authors studied these polynomials, referring to them as {\em descent polynomials}. In particular, they showed that these polynomials have a nonnegative integer coefficient expansion with respect to a natural basis of the ring of polynomials in $n$ \cite[Theorem 3.3]{diaz2019descent}.\footnote{See \cite{JMcC} for recent results involving bounds on the roots of descent polynomials using the Naruse hook-length formula.}

Descent polynomials can equivalently be regarded as polynomials whose value at $n$ gives the number of linear extensions of a ribbon with $n$ elements. Thinking of descent polynomials in these terms, we present a generalization of descent polynomials for free-standing mobile tree posets. 

Let $\mathcal{P}$ be a mobile tree poset. Let $\mathcal{Z}$ be an underlying ribbon of $\mathcal{P}$ with maximal cardinality, and assume for simplicity that $\cP$ is free-standing with respect to $\cZ$.  We write $\{z_1,\ldots,z_k\}$ for the elements of $\mathcal{Z}$, and we require that $\{z_1,\ldots,z_k\}$ are indexed in such a way that there is a cover relation $z_i \lessdot z_j$ or $z_j \lessdot z_i$ if and only if $|i - j| = 1$. 

Define $\mathcal{P}_{\mathcal{Z},m}$ to be the poset whose elements are $\mathcal{P}\sqcup \{z_{k+1}, \ldots, z_{m}\}$ and with the same cover relations as $\mathcal{P}$, plus the additional cover relations $z_{k} \lessdot z_{k+1} \lessdot \cdots \lessdot z_{m-1} \lessdot z_m$ (see Figure~\ref{fig_des_poly_analogue}). We refer to a cover relation of $\mathcal{P}_{\mathcal{Z},m}$ of the form $z_{i+1} \lessdot z_i$ as a {\em descent}. We let $\text{des}(\mathcal{Z})$ denote the number of descents of $\mathcal{Z}$. Additionally, note that $\mathcal{P}_{\mathcal{Z},m}$ is a free-standing mobile tree poset with respect to the ribbon
\[
(\cZ + (z_{k+1} \lessdot \cdots \lessdot z_m)) \oplus \{(z_k,z_{k+1})\}.
\]
If $\cZ$ has at least one descent, let $i = \max\{j \in [m-1] \mid z_{j+1}\lessdot z_j\}$.  Observe that $\mathcal{P}_{\mathcal{Z},m} \ominus \{(z_{i+1},z_{i})\} = \mathcal{P}^\prime+\mathcal{P}^{\prime\prime}$ where $\mathcal{P}^\prime$ (resp.,  $\mathcal{P}^{\prime\prime}$) is a free-standing mobile tree poset with respect to the ribbon on $\{z_1,\ldots, z_i\}$ (resp., $\{z_{i+1}, \ldots, z_{m}\}$), and this ribbon is a ribbon of maximal cardinality. Note that, by the maximality of $i$, the poset $\mathcal{P}^{\prime\prime}$ is a rooted tree.

\begin{figure}
    \centering
\includegraphics[scale=1]{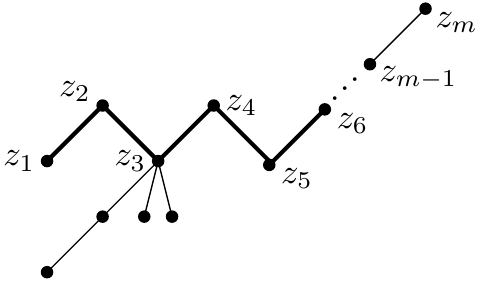}
    \caption{Example of a mobile tree poset $\mathcal{P}_{\mathcal{Z},m}$ where $k = 6$. Here the descents of $\mathcal{Z} = \{z_1,\ldots, z_6\}$ are $z_3\lessdot z_2$ and $z_5 \lessdot z_4.$}
    \label{fig_des_poly_analogue}
\end{figure}

\begin{theorem}
For any $\mathcal{P}_{\mathcal{Z},m}$ as defined above, the number $e(\mathcal{P}_{\mathcal{Z},m})$ is a polynomial in $\mathbb{Q}[N]$, where $N = \#(\mathcal{P}_{\mathcal{Z},m})$. 
If $\mathcal{Z}$ has at least one descent, then the degree of $e(\mathcal{P}_{\mathcal{Z},m})$ is $\#(\mathcal{P}^\prime)$. If $\mathcal{Z}$ has no descents, then the degree of $e(\mathcal{P}_{\mathcal{Z},m})$ is $0$. 
\end{theorem}

\begin{proof}
First, assume that $\mathcal{Z}$ has no descents. It follows that $\mathcal{P}_{\mathcal{Z},m}$ is a rooted tree. This implies that   
\[\begin{array}{rcl}
e(\mathcal{P}_{\mathcal{Z}, m}) & = & \displaystyle \frac{N!}{\prod_{x \in \mathcal{P}_{\mathcal{Z},m}}h_{\cP_{\cZ,m}}(x)}\\
& = & \displaystyle \frac{N!}{(N)(N-1)\cdots ((\#\cP) + 1)\prod_{x \in \mathcal{P}}h_{\cP}(x)}\\
& = & \displaystyle \frac{(\#\mathcal{P})!}{\prod_{x \in \mathcal{P}}h_{\cP}(x)} =  e(\mathcal{P}),
\end{array}\] 
where the second equality follows from the definition of $\mathcal{P}_{\mathcal{Z},m}$.
Therefore, the above expression is a constant polynomial in $\mathbb{Q}[N]$, as desired. 

Now, assume that $\mathcal{Z}$ has at least one descent and that the statement of the theorem holds for all posets $\mathcal{Q}_{\mathcal{Z}^\prime,m^\prime}$, where the number of descents of $\mathcal{Z}^\prime$ is strictly less than the number of descents of $\mathcal{Z}$.  As above, let $i = \max\{j \in [m-1] \mid z_{j+1} \lessdot z_{j}\}.$ By Lemma~\ref{lemma:inclusionexclusion}, we have that 
\begin{align*}
e(\mathcal{P}_{\mathcal{Z},m}) &= e(\mathcal{P}_{\mathcal{Z},m} \ominus \{(z_{i+1},z_{i})\}) - e((\mathcal{P}_{\mathcal{Z},m})_{\{(z_{i+1},z_{i})\}})\\
&= e(\mathcal{P}^\prime+\mathcal{P}^{\prime\prime})  - e((\mathcal{P}_{\mathcal{Z},m})_{\{(z_{i+1},z_{i})\}}).
\end{align*}

Clearly, the maximal ribbons on $\{z_1,\ldots, z_i\}$ and on $\{z_{i+1}, \ldots, z_m\}$ of $\mathcal{P}^\prime$ and $\mathcal{P}^{\prime\prime}$, respectively, have fewer descents than $\mathcal{Z}$. Furthermore, none of the elements of $\{z_1,\ldots, z_i\}$ (resp., $\{z_{i+1},\ldots, z_m\}$) are covered by any element of $\mathcal{P}^\prime\backslash \{z_1,\ldots, z_i\}$ (resp., $\mathcal{P}^{\prime\prime}\backslash \{z_{i+1},\ldots, z_m\}$).
Additionally, by the maximality of $i$, the poset $\mathcal{P}^{\prime\prime}$ is a rooted tree with maximal element $z_{m}$.

Similarly, the poset $(\mathcal{P}_{\mathcal{Z},m})_{\{(z_{i+1},z_{i})\}}$ is a mobile tree poset with respect to the maximal ribbon on the elements $\{z_1,\ldots, z_m\}$, where $z_{i} \lessdot z_{i+1}$. By definition, this maximal ribbon has one fewer descent than $\mathcal{Z}$ does, and there are no elements of this ribbon covered by an element of $(\mathcal{P}_{\mathcal{Z},m})_{\{(z_{i+1},z_{i})\}}\backslash\{z_1,\ldots, z_m\}$.

Applying Proposition~\ref{prop:disjointsum} to the term $e(\cP^\prime + \cP^{\prime\prime})$ in the displayed equation above, we obtain
\begin{align*}
e(\mathcal{P}_{\mathcal{Z},m}) &= \displaystyle \binom{N}{\#(\mathcal{P}^\prime)}e(\mathcal{P}^\prime)e(\mathcal{P}^{\prime\prime}) - e((\mathcal{P}_{\mathcal{Z},m})_{\{(z_{i+1},z_{i})\}})\\
&= \displaystyle \binom{N}{\#(\mathcal{P}^\prime)}e(\mathcal{P}^\prime)\frac{\#(\mathcal{P}^{\prime\prime})}{\prod_{x \in \mathcal{P}^{\prime\prime}}h_{\cP^{\prime\prime}}(x)}  - e((\mathcal{P}_{\mathcal{Z},m})_{\{(z_{i+1},z_{i})\}}),
\end{align*}
where the last equality follows from the fact that $\mathcal{P}^{\prime\prime}$ is a rooted tree. Observe that the factor $e(\mathcal{P}^\prime)$ is independent of $N$. One verifies that $e(\mathcal{P}^{\prime\prime})$ is independent of $N$ using the same calculation that was used in the base case of the induction.  We also have that the binomial coefficient is a polynomial in $N$ of degree $\#(\mathcal{P}^\prime)$. This implies that the first term in this difference is a polynomial in $\mathbb{Q}[N]$ of the desired degree. 

By induction applied to  $e((\mathcal{P}_{\mathcal{Z},m})_{\{(z_{i+1},z_{i})\}})$, we see that the second term in this difference is a polynomial in $\mathbb{Q}[N]$ of degree strictly less than $\#(\mathcal{P}^\prime).$ We obtain that $e(\mathcal{P}_{\mathcal{Z},m})$ is a polynomial in $\mathbb{Q}[N]$ of degree $\#(\mathcal{P}^\prime)$.
\end{proof}

\begin{example}
In the table below, we give examples of the polynomial $e(\mathcal{P}_{\mathcal{Z},m})$ when 
$\mathcal{P}$ is $\mathcal{C}_1(1)$, $\mathcal{C}_1(2)$, and $\mathcal{C}_1(3)$. The posets
$\mathcal{C}_p(n)$ were defined in \eqref{eq: two families of mobiles}.
\end{example}
\begin{table}[!htbp]
\centering
\begin{tabular}{c c c } 
 \hline
 $\mathcal{P}_{\mathcal{Z}, m}$ & polynomial expression for $e(\mathcal{P}_{\mathcal{Z}, m})$  \\ [0.5ex] 
 \hline
$\mathcal({C}_1(1))_{\mathcal{Z},m}$ & 1 \\ 
$\mathcal({C}_1(2))_{\mathcal{Z},m}$ & $\binom{N}{3}-4$  \\
$\mathcal({C}_1(3))_{\mathcal{Z},m}$ & $16\binom{N}{6} - 4\binom{N}{3} + 28$ \\
 [1ex] 
 \hline
\end{tabular}
\end{table}

\begin{remark}
The authors in \cite{diaz2019descent} investigate other questions concerning descent polynomials, including positivity in certain bases and roots of the polynomial (see also the recent work of Jiradilok--McConville \cite{JMcC}). It would be interesting to also study these questions for the polynomial $e(\cP_{\cZ,m})$.
\end{remark}

\subsection{\texorpdfstring{$q$}{q}-analogue of Atkinson's recursive algorithm for linear extensions of any tree poset}
\label{subsec:qatk}

Corollary~\ref{thm:mobiledetdone} gives a determinantal formula to compute the number of linear extensions of mobile tree posets, so this number can be computed in a polynomial number of operations by computing the determinant. This raises the question of whether the number $e(\cP)$ of linear extensions of a tree poset $\cP$ can be computed efficiently. In \cite{atk}, Atkinson gave a recursive quadratic time algorithm to compute $e(\cP)$ for any tree poset $\cP$. This algorithm actually calculates a refinement of $e(\cP)$, which we discuss next.

Let $\cP$ be a poset with $n$ elements.  Given an element $a$ in $\cP$, we define the {\em spectrum of $a$} to be the sequence of nonnegative integers $(\alpha_1,\ldots,\alpha_n)$ where $\alpha_i$ is the number of linear extensions $f$ of $\cP$ with $a$ occurring at position $i$ (that is, with $f^{-1}(i)=a$). Note that $e(\cP) = \sum_{i=1}^n \alpha_i$. Atkinson's algorithm computes, in quadratic time, the spectra of elements of tree posets via a key lemma that shows how to compute the spectra of a poset $\cQ {}^b\backslash_a \cP$ from the spectra of the posets $\cP$ and $\cQ$. In Lemma~\ref{prop:qatk} below, we give an inversion $q$-analogue of this lemma. We were not able to find a major index version of it.

For a labeled $n$-element poset $(\cP,\omega)$ and an element $a$ in $\cP$, we define the {\em $q$-spectrum} of $a$ to be the sequence
$(\alpha_1(q),\ldots,\alpha_n(q))$ of polynomials, where 
\[
\alpha_i(q) := \sum_{\substack{\sigma \in  \mathcal{L}(\cP,\omega)\\\sigma(i) = \omega(a)}} q^{\inv(\sigma)}.
\]
Note that $e^{\inv}_q(\cP) = \sum_{i=1}^n \alpha_i(q)$ and that, for all $i \in [n]$, we have $\alpha_i(1) = \alpha_i$.

\begin{lemma}\label{prop:qatk}
Let $(\cP+\cQ,\omega)$ be a labeled poset where $\cP$ and $\cQ$ have $u$ and $v$ elements, respectively. Suppose that $\omega$ has the property that $\omega(p)<\omega(q)$ for every $p$ in $\cP$ and $q$ in $\cQ$. Let $(\alpha_1(q),\ldots,\alpha_u(q))$ and $(\beta_{1}(q),\ldots,\beta_{v}(q))$ be the $q$-spectra of elements $a$ and $b$ in $\cP$ and $\cQ$, respectively. Then for $r=1,\ldots,u+v$, the $r$th entry of the $q$-spectrum  of $a$ in $(\cQ {}^b\backslash_a \cP,\omega)$  is
\[
\gamma_r(q) = \sum_{i=\max(1,r-v)}^{\min(u,r)}
 \alpha_i(q) q^{(u-i+1)(r-i)} \qbin{r-1}{i-1}{q}\qbin{u+v-r}{u-i}{q} \sum_{j=r-i+1}^v \beta_j(q).
 \]
\end{lemma}

\begin{proof}[Proof]
A linear extension $f$ of $(\cQ {}^b\backslash_a \cP,\omega)$ with $a$ in position $r$ (that is, with $f(r)=\omega(a)$) is obtained by combining a linear extension $g$ of $(\cP,\omega)$ with $a$ in position $i$ for $\max(1,r-v) \leq i \leq \min(u,r)$ and a linear extension $h$ of $(\cQ,\omega)$ with $b$ in some position $j$ with $r-i+1\leq j \leq v$ as follows: (i) choose  $i-1$ positions $S_1$ out of the $r-1$ positions before $a$ in $f$ to insert, in order, the entries of $g$ before $a$; (ii) choose $u-i$ positions $S_2$ out of the $u+v-r$ positions after $a$ in $f$ to insert, in order, the entries of $g$ after $a$; (iii) insert $a$ in position $r$ of $f$; and (iv) insert the entries of $h$ in order in the remaining positions of $f$.

Next, we calculate the inversions of $f$ in terms of the inversions of $g$ and $h$. Each inversion of $g$ and $h$ is an inversion of $f$. It remains to count the inversions coming from pairs consisting of an element of $\cQ$ preceding an element of $\cP$, since the labels of $\cP$ are smaller than those of $\cQ$. There are three kinds of these inversion pairs in $f$: both elements are before $a$, both elements are after $a$, and one of the $r-i$ elements of $\cQ$ before $a$ with one of the $u-i+1$ elements of $\cP$ after and including $a$.  The number of inversions of each of these kinds is $\inv(S_1)$, $\inv(S_2)$, and $(u-i+1)(r-i)$, respectively. (Here, the inversions of a set $S\subset[n]$ are calculated as the inversions of the binary word of length $n$ corresponding to the positions of the set $S$.) Thus, in total, we have 
\[
\inv(f) \,=\, \inv(g) + \inv(h) + \inv(S_1) + \inv(S_2) + (u-i+1)(r-i).
\]
We then consider the contribution to $\gamma_r(q)$ of all the linear extensions $f$ obtained from fixed linear extensions $g$ and $h$. Then, the result follows by using a well-known expansion of $q$-binomial coefficients \cite[Prop. 1.7.1]{EC1}.
\end{proof}

In order to use the $q$-Atkinson algorithm on a tree poset $\cT$, a labeling is needed so that when we iteratively delete certain cover relations to split $\cT$ into two posets $\cP$ and $\cQ$, we have the property that all elements of $\cP$ have smaller labels than those of $\cQ$.  Such labelings exist, and an algorithm provided in \cite{grosser} produces them.

\bibliography{references}
\bibliographystyle{plain}
\end{document}